\documentclass[10pt, a4paper]{article}
\usepackage[T1]{fontenc}

\usepackage{amsmath,amsthm, amssymb,color}
\usepackage{mathtools}
\usepackage{esint}
\usepackage{hyperref}
\usepackage{cite}
\usepackage{etoolbox}
\usepackage{float}

\usepackage{fancyhdr}
\usepackage[margin=3cm]{geometry}

\numberwithin{equation}{section}
\newtheorem{theorem}{Theorem}[section]
\newtheorem{proposition}[theorem]{Proposition}
\newtheorem{lemma}[theorem]{Lemma}

\newtheorem{remark}[theorem]{Remark}
\theoremstyle{definition}
\newtheorem{definition}[theorem]{Definition}

\newcommand{\writefoot}[1]{
    \renewcommand{\thefootnote}{}
    \footnotetext{\hspace{-16.5pt}\scriptsize#1}
    \renewcommand{\thefootnote}{\arabic{footnote}}
}
\numberwithin{equation}{section}

\begin{document}

\writefoot{\small \textbf{AMS subject classifications (2020).} Primary: 
 34A34, % (1980-now) Nonlinear ordinary differential equations and systems 
 34D05; % (1973-now) Asymptotic properties of solutions to ordinary differential equations 
Secondary:
 92D25. % (1991-now) Population dynamics (general) 
 \smallskip}
\writefoot{\small \textbf{Keywords.} Lotka-Volterra system, three-species competition, non-periodic oscillatory behavior, asymptotic analysis. 
\smallskip
}
\writefoot{\small \textbf{Acknowledgements:} 
QVT acknowledges the  support of Région Normandie through the Ph.D fellowship program.
AD and QG are supported by the project TALEDZ of the research program PREZODE and received government funding managed by the Agence Nationale de la Recherche under the France 2030 program, reference  ANR-24-PEPZ-0003. 
QG was partially supported by a PEPS-JCJC grant from CNRS (2019) and by ANR grant “Indyana” Number ANR-21-CE40-0008.}

\writefoot{\small$\,^*$ Corresponding author. e-mail: {\tt quang-vinh.tran@univ-lehavre.fr}}
\begin{center}
\begin{minipage}{0.9\textwidth}
\centering
\LARGE{\bf Heteroclinic cycles and long-time behavior in a
singularly perturbed three-species Lotka–Volterra system}\bigskip
\end{minipage}

\Large
Arnaud Ducrot$^{a}$, Quentin Griette$^{a}$, and Quang-Vinh Tran$^{a,*}$\bigskip \\
\normalsize
$^{a}$Université Le Havre Normandie, Normandie Univ.,
LMAH UR 3821, F-76600 Le Havre, France.\\
\bigskip

\end{center}

\begin{abstract}
%Short description about the model, describe the result, describe the difficulty, why is interesting.

We study a three-species competitive Lotka--Volterra system with cyclic dominance, motivated by rock--paper--scissors interactions in ecological communities. The model contains a singularly strong interaction coefficient of order $1/\varepsilon$, representing a regime in which one species exerts a much stronger competitive pressure on another. We show that, despite this strong asymmetry, the cyclic dominance mechanism persists.

Our main result is a quantitative description of one complete dominance pattern in the singular regime $0<\varepsilon\ll 1$. Starting near the state where the third species dominates, we prove that the solution successively passes near the dominance states of the second and first species before returning to a comparable configuration. We obtain logarithmic estimates for the return time and for the population densities at the return section. In particular, although one species becomes extremely small at the beginning and at the end of each pattern, it recovers to an order-one density during the cycle.

We then prove that this pattern repeats indefinitely and that the omega-limit set is composed by three heteroclinic orbits on the boundary of the positive orthant, connecting the three single-species equilibria. Thus the long-time behavior is cyclic but aperiodic, with increasingly long excursions near successive single-species states.
\end{abstract}

\section{Introduction}%\label{sec:intro}
We consider the following three-species competitive Lotka--Volterra system
\begin{align}\label{eq:main-1}
\begin{split}
\begin{cases}
v_{1}'(t)=v_{1}(t)\bigl(1-v_{1}(t)-k_{12}v_{2}(t)-k_{13}v_{3}(t)\bigr),\\[0.3em]
v_{2}'(t)=v_{2}(t)\bigl(1-\tfrac{1}{\varepsilon}\,v_{1}(t)-v_{2}(t)-k_{23}v_{3}(t)\bigr),\\[0.3em]
v_{3}'(t)=v_{3}(t)\bigl(1-k_{31}v_{1}(t)-k_{32}v_{2}(t)-v_{3}(t)\bigr),
\end{cases}
\end{split}
\end{align}
where $\varepsilon>0$ is a small parameter and $k_{ij}>0$ are fixed interaction coefficients.
Throughout the paper, we work in the cyclic dominance regime
\begin{equation}\label{eq:cond-sys-0}
0<k_{12},k_{23},k_{31}<1,\qquad k_{13},k_{32}>1,\qquad k_{21}=\tfrac1\varepsilon,
\end{equation}
together with the additional conditions \eqref{eq:cond-sys-2} stated below.

Lotka--Volterra systems constitute one of the classical mathematical
frameworks for population dynamics, with important applications in ecology, especially in the study of interacting populations and competitive systems. These systems go back to the works of
Lotka~\cite{L20a,L20b,L25} and Volterra~\cite{V26,V26b}, and were later
supported experimentally by Gause~\cite{G34,GW35}. In the two-species competitive case, the long-time behavior is now relatively well understood: depending on the interaction coefficients, solutions typically converge either to a coexistence equilibrium or to a competitive-exclusion equilibrium; see, for example, \cite{G76,M02}. In higher dimensions, this simple picture no longer describes all possible behaviors. The invariant boundary of the positive orthant may play an important role in the dynamics, and questions related to extinction, persistence, and the global organization of attracting sets become more delicate. These issues have been studied from several complementary viewpoints, including the theory of competitive and monotone dynamical systems~\cite{HS06}, local-to-global analysis in competitive Lotka--Volterra systems~\cite{ZZ03}, and persistence/extinction theory~\cite{HSG79,FW85,BW88,Z95,AS05}.

The situation is already substantially richer for three-species systems. In contrast with the two-species case, three-dimensional competitive Lotka--Volterra systems may exhibit non-equilibrium long-time behavior, including convergence toward heteroclinic cycles and, in special parameter regimes, periodic dynamics. Periodic orbits and Hopf bifurcations in three-dimensional competitive Lotka--Volterra systems were studied by Zeeman~\cite{Z93,Z96}, while van den Driessche and Zeeman~\cite{vddZ98} gave criteria excluding periodic orbits in certain parameter regimes. The global dynamics of three-species competitive Lotka--Volterra systems, including permanence, impermanence, and boundary dynamics, continue to be an active topic of research; see, for instance, \cite{NX24,NWX25}.

In~\cite{HSG79}, the authors studied competitive three-species Lotka--Volterra systems through the pairwise outcomes of the corresponding two-species subsystems. Among the possible configurations, a particularly important one is the non-transitive case, where species~1 outcompetes species~2, species~2 outcompetes species~3, and species~3 outcompetes species~1. In this case, the three species interact in a cyclic way, leading to rock--paper--scissors type dynamics and making the long-time behavior of the full system especially delicate to analyze.

This phenomenon was analyzed by May and Leonard~\cite{ML75} in the symmetric case $k_{12}=k_{23}=k_{31}=\alpha$ and $k_{13}=k_{21}=k_{32}=\beta$. In their model, the stability of the coexistence equilibrium is determined by the threshold $\alpha+\beta=2$: when $\alpha+\beta<2$, the coexistence equilibrium is asymptotically stable; when $\alpha+\beta=2$, the system admits closed orbits around the coexistence equilibrium; and when $\alpha+\beta>2$, the coexistence equilibrium becomes unstable and the dynamics is organized by a heteroclinic cycle on the boundary. In the latter regime, trajectories successively pass near states where one species dominates, then move toward the next dominant state, and eventually return close to the initial one. {Moreover, the time needed to complete each cycle increases along the trajectory.} Related asymptotic questions for Lotka--Volterra competing-species systems were also studied by Coste, Peyraud, and Coullet~\cite{CPC79}, who in particular gave an analytic description of the closed orbits observed by May and Leonard in the equal-growth-rate three-species case. More recently, {heteroclinic cycles} in Lotka--Volterra systems have also been revisited from stochastic, theoretical, and computational viewpoints; see, for instance, Barendregt and Thomas~\cite{BT23} and Bortolan, Kalita, Langa, and Moura~\cite{BKLM25}.

Schuster, Sigmund, and Wolff~\cite{SSW79} subsequently gave a more precise description of this asymptotic behavior in terms of \(\omega\)-limit sets. In particular, they clarified the long-time structure of the May--Leonard dynamics and showed that similar cyclic asymptotic behavior may also occur in asymmetric three-species competitive systems.

A detailed analysis of the asymmetric May--Leonard model was later provided by Chi, Hsu, and Wu~\cite{CWH98}. In their notation, setting $A_i=1-\alpha_i$ and $B_i=\beta_i-1$, the relevant classification parameter is $\mathcal H:=A_1A_2A_3-B_1B_2B_3$. They obtained a complete classification according to the sign of $\mathcal H$. If $\mathcal H>0$, the unique interior equilibrium is globally asymptotically stable in the interior of the positive orthant. If $\mathcal H<0$, the interior equilibrium is a saddle point with a one-dimensional stable manifold, and solutions with initial data outside this manifold oscillate aperiodically around the boundary heteroclinic cycle. If $\mathcal H=0$, the system lies in the borderline case associated with a Hopf bifurcation and admits a family of neutrally stable periodic solutions. Moreover, nontrivial periodic solutions exist only in the borderline case $\mathcal H=0$. In our variables,
\begin{align*}
\mathcal{H}_{\varepsilon}=(1-k_{12})(1-k_{23})(1-k_{31})-(k_{13}-1)\left(\frac{1}{\varepsilon}-1\right)(k_{32}-1).
\end{align*}
For $\varepsilon>0$ sufficiently small, one has $\mathcal H_{\varepsilon}<0$ under our assumptions. Hence the system lies in the heteroclinic regime of Chi--Hsu--Wu. In particular, the cyclic behavior studied here is not generated by an interior periodic orbit; rather, it is organized by a heteroclinic cycle on the boundary.

Such cyclic mechanisms also appear in biological systems. For instance, Sinervo and Lively~\cite{SL96} studied a rock--paper--scissors mechanism in the side-blotched lizard, where three male strategies dominate one another in a cyclic way. Kirkup and Riley~\cite{KR04} later studied another rock--paper--scissors mechanism in vivo, involving three strains of \emph{E. coli}, usually denoted by $C$, $R$, and $S$. In this bacterial system, the three strains also dominate one another cyclically. Cyclic dominance has also been studied in stochastic eco-evolutionary models of clonal interference. In particular, Billiard and Smadi~\cite{BS17} considered a three-type birth-and-death model with type-dependent competition, allowing non-transitive interactions between competing mutations and leading to patterns such as coexistence or rock--paper--scissors cycles. These examples show that cyclic dominance is a biologically relevant mechanism. More broadly, rock--paper--scissors type dynamics also play a central role in evolutionary game theory; see, for instance, Hofbauer and Sigmund~\cite{HS98}.

{These examples} motivates our study of system~\eqref{eq:main-1} under conditions~\eqref{eq:cond-sys-0}. We focus first on one complete pattern of the dynamics, in order to understand how the three species evolve during one cycle. The main feature of the model is the coefficient of order $1/\varepsilon$, with $0<\varepsilon\ll1$, which represents a very strong competitive effect of
species~1 on species~2. A priori, such a singular interaction could destroy the cyclic mechanism by forcing species~2 to remain negligible. Our analysis shows that this is not the case.

Our first objective is to derive a detailed description of the solution over one complete pattern. We start from an initial configuration located on the section $v_1=\varepsilon$, with $v_2(0)$ sufficiently small and $v_3(0)$ close to $1$. This corresponds to a state near the boundary equilibrium $\mathcal E_3=(0,0,1)$, where the first species is at the singular scale
$\varepsilon$, the second species is very rare, and the third species is dominant. We then follow the trajectory until it returns to the same section $v_1=\varepsilon$, but with negative velocity. More precisely, we define the {\it pattern time} by
\begin{align*}
T\coloneqq\inf\{t>0:\ v_1(t)=\varepsilon,\ v_1'(t)<0\}.
\end{align*}
The condition $v_1'(T)<0$ ensures that the trajectory crosses the section in the same direction as at the beginning of the next pattern. During the interval $[0,T]$, we prove that the solution successively enters neighborhoods of $\mathcal E_2=(0,1,0)$ and $\mathcal E_1=(1,0,0)$, and then returns to a configuration of the same type as the initial one, namely $v_1=\varepsilon$, with $v_2$ small and $v_3$ close to $1$. This completes one full pattern of the cyclic dominance mechanism. 

This behavior is consistent with the classical picture described by May and Leonard~\cite{ML75}, in which the solution passes successively near states where one species dominates while the others remain at very low density. We further show that this mechanism can be repeated, and that the time needed to complete each new pattern becomes longer and longer, while the trajectory approaches the boundary equilibria more and more closely. In this sense, our results recover the qualitative cyclic behavior described in the work of May and Leonard, while providing a quantitative description of one complete pattern in the singular regime considered here.

We now summarize the main results of the paper. First, we prove a quantitative result for one complete pattern of the dynamics. Starting from the section $v_1=\varepsilon$, with $v_2$ small and $v_3$ close to $1$, we show that the solution returns in finite time to the same section with $v_1'<0$. During this motion, the trajectory first moves away from the neighborhood of $\mathcal E_3$, then passes near $\mathcal E_2$, reaches a neighborhood of $\mathcal E_1$, and finally comes back to a state of the same type as the initial one. We also obtain explicit logarithmic bounds on the return time and
on the state of the solution at the return section, in terms of $\varepsilon$ and the initial size of $v_2$.

An important point is that species~2 does not disappear during this pattern. Even though species~1 has a very strong negative effect on species~2 because the coefficient $1/\varepsilon$ is large, species~2 still grows to a fixed positive size. More precisely, there exists a time $\hat T\in(0,T)$ such that
\begin{align*}
v_2(\hat T)\ge k_0,
\qquad
k_0=\frac12\Bigl(1-k_{23}+\frac1{k_{32}}\Bigr)>0.
\end{align*}
Thus $v_2$ is small at the beginning and also small when the solution returns to the section, but it becomes of order one during the pattern. This shows that the motion is truly cyclic.

Second, we show that this pattern can be repeated. For large times, the
solution visits neighborhoods of the three boundary equilibria
\begin{align*}
\mathcal{E}_3=(0,0,1),\qquad \mathcal{E}_2=(0,1,0),\qquad \mathcal{E}_1=(1,0,0),
\end{align*}
in the cyclic order $\mathcal{E}_3\to \mathcal{E}_2\to \mathcal{E}_1\to \mathcal{E}_3$. The estimates for one pattern
also explain why the next patterns take longer and longer: after each return, the value of $v_2$ becomes smaller, so the solution needs more time to grow again.

{Finally, we prove that these three equilibria are connected by a heteroclinic cycle
\begin{align*}
\Upsilon=\Upsilon_{32}\cup\Upsilon_{21}\cup\Upsilon_{13}
\cup\{\mathcal E_1,\mathcal E_2,\mathcal E_3\}.
\end{align*}
Here $\Upsilon_{32}$ connects $\mathcal E_3$ to $\mathcal E_2$,
$\Upsilon_{21}$ connects $\mathcal E_2$ to $\mathcal E_1$, and
$\Upsilon_{13}$ connects $\mathcal E_1$ back to $\mathcal E_3$. We then
show that the solution gets closer and closer to this cycle as time goes to infinity, and that its omega-limit set is exactly this cycle: $\omega(v)=\Upsilon$. Therefore, the long-time behavior is not a periodic orbit. It is a cyclic but aperiodic motion organized by a heteroclinic cycle.}

The paper is organized as follows. In Section~2, we state the assumptions and the main results. Section~3 collects preliminary properties of the system. Section~4 is devoted to the proof of the one-pattern theorem. In Section~5, we prove the repeated-visit property, construct the heteroclinic cycle, establish the local shadowing of the heteroclinic branches, and identify the omega-limit set.

%What is known for three species case?---Done.

%Speak about why we choose this specific parameters?---Done.

%What is our goal?---Done.

%Summary about the main results.--Done.

%Present the figure Phase at this point.---Done.
\section{Main results}
\subsection{Rescaled formulation and notation}
We perform the change of variables
\begin{align*}
u_1=\frac{v_1}{\varepsilon},\qquad u_2=v_2,\qquad u_3=v_3.
\end{align*}
Equivalently,
\begin{align*}
v_1=\varepsilon u_1,\qquad v_2=u_2,\qquad v_3=u_3.
\end{align*}
We denote by
\begin{align*}
\Phi_\varepsilon(u_1,u_2,u_3):=(\varepsilon u_1,u_2,u_3)
\end{align*}
the corresponding linear rescaling. The map
$\Phi_\varepsilon$ is a linear homeomorphism of the positive orthant onto
itself. In particular, the time variable is not changed, and the dynamics in
the $u$-variables are equivalent to those in the $v$-variables after this
rescaling of the first component.

In terms of $u=(u_1,u_2,u_3)$, system \eqref{eq:main-1} becomes
\begin{align}
\label{eq:system}
\begin{cases}
u_1'(t)=u_1(t)\bigl(1-\varepsilon u_1(t)-k_{12}u_2(t)-k_{13}u_3(t)\bigr),\\[0.3em]
u_2'(t)=u_2(t)\bigl(1-u_1(t)-u_2(t)-k_{23}u_3(t)\bigr),\\[0.3em]
u_3'(t)=u_3(t)\bigl(1-\varepsilon k_{31}u_1(t)-k_{32}u_2(t)-u_3(t)\bigr).
\end{cases}
\end{align}
Under this change of variables, the section $v_1=\varepsilon$ becomes
$u_1=1$. Hence the pattern time is written as
\begin{align*}
T:=\inf\left\{t>0:\ u_1(t)=1,\ u_1'(t)<0\right\}.
\end{align*}
{We now clarify the relation between the geometric objects in the original variables and those in the rescaled variables. In the original variables, the boundary equilibria are denoted by
\begin{align*}
\mathcal E_1=(1,0,0),\qquad
\mathcal E_2=(0,1,0),\qquad
\mathcal E_3=(0,0,1).
\end{align*}
In the rescaled variables, these equilibria become
\begin{align*}
E_1=\Phi_\varepsilon^{-1}(\mathcal E_1)
=\left(\frac1\varepsilon,0,0\right),\qquad
E_2=\Phi_\varepsilon^{-1}(\mathcal E_2)=(0,1,0),\qquad
E_3=\Phi_\varepsilon^{-1}(\mathcal E_3)=(0,0,1).
\end{align*}
We denote by
\begin{align*}
\Gamma=\Gamma_{32}\cup\Gamma_{21}\cup\Gamma_{13}
\cup\left\{E_1,E_2,E_3\right\},
\end{align*}
the heteroclinic cycle of the rescaled system \eqref{eq:system}. The
corresponding heteroclinic cycle in the original variables is denoted by
\begin{align*}
\Upsilon=\Upsilon_{32}\cup\Upsilon_{21}\cup\Upsilon_{13}
\cup\left\{\mathcal E_1,\mathcal E_2,\mathcal E_3\right\},
\end{align*}
where $\Upsilon_{ij}:=\Phi_\varepsilon(\Gamma_{ij})$ and $\mathcal E_i=\Phi_\varepsilon(E_i)$. Equivalently, $\Upsilon=\Phi_\varepsilon(\Gamma)$. Thus $\Upsilon$ and $\Gamma$ describe the same heteroclinic cycle in the two coordinate systems.

In the rescaled variables, the heteroclinic connections are oriented in the cyclic order
\begin{align*}
E_3\longrightarrow E_2\longrightarrow E_1\longrightarrow E_3.
\end{align*}
That is, $\Gamma_{32}$ connects $E_3$ to $E_2$, $\Gamma_{21}$ connects $E_2$ to $E_1$, and $\Gamma_{13}$ connects $E_1$ back to $E_3$.}

From this point on, unless otherwise stated, we work in the rescaled variables
and use the notation $E_i$ and $\Gamma$.

The cyclic dominance assumptions on the coefficients are rewritten as
\begin{equation}
\label{eq:cond-sys}
0<k_{12},k_{23},k_{31}<1,\qquad k_{13},k_{32}>1,\qquad k_{21}=1.
\end{equation}
In addition, we impose the structural conditions
\begin{align}
\label{eq:cond-sys-2}
k_{13}(1-k_{31})>1,\qquad
k_{32}(1-k_{23})>1,\qquad
k_{13}+k_{31}>2,\qquad
k_{32}+k_{23}>2,
\end{align}
and the smallness condition
\begin{align}
\label{eq:cond-sys-2-eps}
0<\varepsilon<\min\left\{k_{12},1-k_{12},\frac{1}{2-k_{12}}\right\}.
\end{align}
These assumptions ensure that the thresholds used in the pattern analysis are
strictly positive.
\subsection{Statement of main theorems}
We first introduce some notation used throughout the statement of the main results.
\begin{definition}
Let $X$ be a set and $f, g$ be two functions from $X$ to $\mathbb{R}$. We write $g(t)\lesssim f(t)$ to denote that there exists a positive constant $C>0$, which may depend on the parameters $(k_{ij})$ only (and not on $\varepsilon$), such that $g(t)\leq Cf(t)$ for all $t\in X$. We will use the same notation for constants (that can be considered as functions on the interval $X=\{0\}$).
\end{definition}
For any $x=(x_1,x_2,x_3)\in\mathbb{R}^3$, we denote by \begin{align*} \|x\|_\infty:=\max\{|x_1|,|x_2|,|x_3|\}, \end{align*} the usual sup-norm on $\mathbb{R}^3$.

We recall that the omega-limit set of a forward orbit \(u(t)\) is defined by
\begin{align*}
\omega(u)
:=
\left\{
y\in\mathbb R^3:\ \exists\, t_n\to+\infty
\text{ such that } u(t_n)\to y
\right\}.
\end{align*}
We now state our main results.
\begin{theorem}\label{thm:main}
Suppose that the coefficients $(k_{ij})$ satisfy \eqref{eq:cond-sys}, \eqref{eq:cond-sys-2}--\eqref{eq:cond-sys-2-eps}; let $0<\delta_1^-<\delta_1^+$ and $\delta_2>0$ be given,  and let $u=(u_1,u_2,u_3)$ denote the solution of \eqref{eq:system}
with initial data
\begin{align}\label{eq:IC}
u_1(0)=1,\qquad \delta_1^-<u_2(0)<\delta_1^+,\qquad 1-\delta_2<u_3(0)<1.
\end{align}
There exist constants $\delta_2^0,\varepsilon_0\in(0,1]$ such that, for every
$0<\varepsilon\le \varepsilon_0$, there exists $0<\delta_1^0(\varepsilon)\le 1$
with the following property: if
$\delta_1^+\le \delta_1^0(\varepsilon)$ and $\delta_2\le \delta_2^0$,
then the pattern time
\begin{align}\label{eq:pattern-time}
T=\inf\{t>0:\ u_1(t)=1 \ \text{and}\ u_1'(t)<0\},
\end{align}
is well-defined and finite, and
\begin{align}\label{eq:est-u-main}
\begin{cases}
\ln\delta_1^- - \lambda^- \ln\varepsilon
\lesssim
\dfrac{\ln\big(u_2(T)\big)}{1+\tfrac1\varepsilon}
\lesssim
\ln\delta_1^+ - \lambda^+ \ln\varepsilon,\\
\ln\big(1-u_3(T)\big)\lesssim \ln\varepsilon,
\end{cases}
\end{align}
for some constants $\lambda^\pm$ satisfying $0<\lambda^+<\lambda^-\lesssim1$. Moreover, there exists $\hat T\in(0,T)$ such that
\begin{align}\label{eq:est-u2}
u_2(\hat T)\ge k_0\coloneqq \frac12\Bigl(1-k_{23}+\frac1{k_{32}}\Bigr)>0,
\end{align}
and the pattern time $T$ and the maximum of $u_1$ satisfy respectively
\begin{align}\label{eq:est-That-T}
-\ln\delta_{1}^{+}+\kappa\ln\varepsilon
\lesssim T\lesssim -\ln\delta_{1}^{-}-\ln\varepsilon,\quad  
\text{ and } \quad 
%\end{align}
%and
%\begin{align}%\label{eq:est-max-u1}
\tfrac{1}{\varepsilon}\lesssim\max_{t\in[0,T]}u_1(t)<\tfrac{1}{\varepsilon},
\end{align}
for some constant $0<\kappa\lesssim 1$ depending only on $(k_{ij})$.
\end{theorem}
Theorem~\ref{thm:main} gives a quantitative description of one complete
pattern of the dynamics. Starting from the section $u_1=1$, with $u_2(0)$ small and $u_3(0)$ close to $1$, the solution returns in finite time to the same section $u_1=1$ with $u_1'(T)<0$. At this return time, the second component $u_2(T)$ is still small, with explicit logarithmic upper and lower bounds, while the third component $u_3(T)$ remains close to $1$. The theorem also shows that the pattern is nontrivial: although $u_2$ is small both at the beginning and at the return time, {it grows to a fixed positive size during the pattern}, namely $u_2(\hat T)\ge k_0>0$ for some $\hat T\in(0,T)$. Moreover, the first component reaches the singular
scale $1/\varepsilon$, and the duration of the pattern is controlled
quantitatively in terms of $\delta_1^\pm$ and $\varepsilon$.
\begin{proposition}\label{thm:corner-visits}
Assume that the coefficients $(k_{ij})$ satisfy \eqref{eq:cond-sys},
\eqref{eq:cond-sys-2}--\eqref{eq:cond-sys-2-eps}. Let
$\delta_2^0,\varepsilon_0$ and $\delta_1^0(\varepsilon)$ be given by
Theorem~\ref{thm:main}. Up to reducing $\delta_1^0(\varepsilon)$ if necessary,
we still denote by $\delta_1^0(\varepsilon)$ the corresponding threshold.
Let $\xi\in(0,1)$ be arbitrary.

For every $0<\varepsilon\le \varepsilon_0$, every
$0<\delta_2\le \delta_2^0$, and every initial datum satisfying
\begin{align*}
u_1(0)=1,\qquad 0<u_2(0)\le \delta_1^0(\varepsilon),\qquad 1-\delta_2\le u_3(0)<1,
\end{align*}
let $u=(u_1,u_2,u_3)$ denote the corresponding solution of \eqref{eq:system}.
Define recursively
\begin{align*}
t_0:=0,\qquad
t_{n+1}:=\inf\{t>t_n:\ u_1(t)=1 \text{ and } u_1'(t)<0\}.
\end{align*}
Then there exists an integer $N_\xi\ge 1$ such that, for every $n\ge N_\xi$,
there exist times
\begin{align*}
\tau_{3,n}^\xi<\tau_{2,n}^\xi<\tau_{1,n}^\xi
\qquad\text{in }(t_n,t_{n+1}),
\end{align*}
for which
\begin{align*}
\left\|u(\tau_{3,n}^\xi)-E_{3}\right\|_\infty\le \xi,
\qquad
\left\|u(\tau_{2,n}^\xi)-E_{2}\right\|_\infty\le \xi,
\qquad
\left\|u(\tau_{1,n}^\xi)-E_{1}\right\|_\infty\le \xi.
đ\end{align*}
\end{proposition}
Proposition~\ref{thm:corner-visits} turns the one-pattern description of Theorem~\ref{thm:main} into a repeated-visit property. It shows that the return construction can be iterated and that, for every prescribed neighborhood size $\xi>0$, all sufficiently large patterns visit the three boundary equilibria successively in the order $E_3\to E_2\to E_1$. This is the first global manifestation of the cyclic dominance mechanism and prepares the construction of the heteroclinic cycle in the next theorem.

The following results describe the asymptotic structure generated by the
one-pattern mechanism. Theorem~\ref{thm:heteroclinic-connections} first proves
the existence of three heteroclinic connections joining the equilibria in the
cyclic order
\begin{align*}
E_3\longrightarrow E_2\longrightarrow E_1\longrightarrow E_3.
\end{align*}
We denote by
\begin{align*}
\Gamma
:=
\Gamma_{32}\cup\Gamma_{21}\cup\Gamma_{13}\cup\{E_1,E_2,E_3\},
\end{align*}
the corresponding heteroclinic cycle.

\begin{theorem}[Existence of the heteroclinic cycle]
\label{thm:heteroclinic-connections}
Assume that the coefficients $(k_{ij})$ satisfy \eqref{eq:cond-sys},
\eqref{eq:cond-sys-2}--\eqref{eq:cond-sys-2-eps}. Fix
$0<\varepsilon\le \varepsilon_0$. Then there exist three global solutions
\begin{align*}
\gamma^{32},\qquad \gamma^{21},\qquad \gamma^{13}
\in C^1(\mathbb R;\mathbb R_+^3)
\end{align*}
of \eqref{eq:system}. For $(i,j)\in\{(3,2),(2,1),(1,3)\}$, we write
\begin{align*}
\gamma^{ij}(t)
=
\bigl(\gamma^{ij}_1(t),\gamma^{ij}_2(t),\gamma^{ij}_3(t)\bigr),
\qquad t\in\mathbb R,
\end{align*}
so that $\gamma^{ij}_\ell$ denotes the $\ell$-th component of $\gamma^{ij}$.
These solutions satisfy
\begin{align*}
\gamma^{32}_1\equiv0,\qquad
\gamma^{21}_3\equiv0,\qquad
\gamma^{13}_2\equiv0,
\end{align*}
and
\begin{align*}
\gamma^{32}(t)\to E_3\quad\text{as }t\to-\infty,
\qquad
\gamma^{32}(t)\to E_2\quad\text{as }t\to+\infty,
\end{align*}
\begin{align*}
\gamma^{21}(t)\to E_2\quad\text{as }t\to-\infty,
\qquad
\gamma^{21}(t)\to E_1\quad\text{as }t\to+\infty,
\end{align*}
\begin{align*}
\gamma^{13}(t)\to E_1\quad\text{as }t\to-\infty,
\qquad
\gamma^{13}(t)\to E_3\quad\text{as }t\to+\infty.
\end{align*}
In particular, the system \eqref{eq:system} admits the heteroclinic cycle
\begin{align*}
E_3\longrightarrow E_2\longrightarrow E_1\longrightarrow E_3.
\end{align*}
\end{theorem}

Theorem~\ref{thm:heteroclinic-connections} identifies the geometric object organizing the long-time dynamics. It proves the existence of three heteroclinic connections joining $E_3$ to $E_2$, $E_2$ to $E_1$, and $E_1$ back to $E_3$. These three connections form the heteroclinic cycle $\Gamma$.

For each pattern interval $[t_n,t_{n+1}]$, we denote by
$T_{1,n}$, $T_{2,n}$, and $T_{33,n}$ the corresponding transition times
defined as in the proof of Theorem~\ref{thm:main}.

{The following two results describe how the one-pattern dynamics given by Theorem~\ref{thm:main} organizes the global asymptotic behavior. The first one gives a local shadowing description near each heteroclinic branch, after suitable time shifts. The second one identifies the omega-limit set of the solution.}

\begin{theorem}\label{thm:assymp-sol-hetero}
Under the assumptions of Theorem~\ref{thm:main}, let \(u\) be the corresponding
solution of \eqref{eq:system}. Let
\(\Gamma_{32},\Gamma_{21},\Gamma_{13}\) be the three heteroclinic orbits given
by Theorem~\ref{thm:heteroclinic-connections}. Then, for every \(A>0\),
\begin{align}
\sup_{|s|\le A}
d\bigl(u(T_{1,n}+s),\Gamma_{32}\bigr)
&\xrightarrow[n\to+\infty]{}0, \label{eq:geom-conv-32}\\
\sup_{|s|\le A}
d\bigl(u(T_{2,n}+s),\Gamma_{21}\bigr)
&\xrightarrow[n\to+\infty]{}0, \label{eq:geom-conv-21}\\
\sup_{|s|\le A}
d\bigl(u(T_{33,n}+s),\Gamma_{13}\bigr)
&\xrightarrow[n\to+\infty]{}0. \label{eq:geom-conv-13}
\end{align}
Here $d(x,\Gamma):=\inf_{y\in\Gamma}\|x-y\|_\infty$.
\end{theorem}
\allowdisplaybreaks
\begin{theorem}\label{thm:omega-limit-cycle}
Under the assumptions of Theorem~\ref{thm:main}, let $u$ be the corresponding
solution of \eqref{eq:system}. Let $\Gamma_{32},\Gamma_{21},\Gamma_{13}$ are the three heteroclinic orbits
given by Theorem~\ref{thm:heteroclinic-connections}. Then
\begin{align*}
d(u(t),\Gamma)
:=
\inf_{y\in\Gamma}\|u(t)-y\|_\infty
\xrightarrow[t\to+\infty]{}0.
\end{align*}
Moreover, the omega-limit set of \(u\), is exactly the heteroclinic cycle, i.e. $\omega(u)=\Gamma$.
\end{theorem}
{Theorem~\ref{thm:assymp-sol-hetero} gives a local shadowing description of the orbit near each heteroclinic branch. More precisely, after shifting time around the intermediate times $T_{1,n}$, $T_{2,n}$, and $T_{33,n}$, the trajectory converges locally, in distance, to the corresponding heteroclinic branches
$\Gamma_{32}$, $\Gamma_{21}$, and $\Gamma_{13}$. Thus, on every fixed time window, each sufficiently large pattern follows the appropriate branch of the heteroclinic cycle.

Finally, Theorem~\ref{thm:omega-limit-cycle} identifies the global long-time limit set of the solution: the distance from $u(t)$ to $\Gamma$ tends to zero as $t\to+\infty$, and the omega-limit set satisfies $\omega(u)=\Gamma$. Thus the heteroclinic cycle is the geometric object governing the asymptotic dynamics of the system.
\subsection{Numerical experiments}
In this section, we present several numerical experiments to support and illustrate the theoretical results obtained above.  We fix $k_{12}=0.1$, $k_{13}=4$, $k_{23}=0.1$, $k_{31}=0.1$, $k_{32}=5$; all parameter values are given in Appendix \ref{sec:parameters}.}

\begin{figure}[H]
\centering
\includegraphics[width=0.8\linewidth]{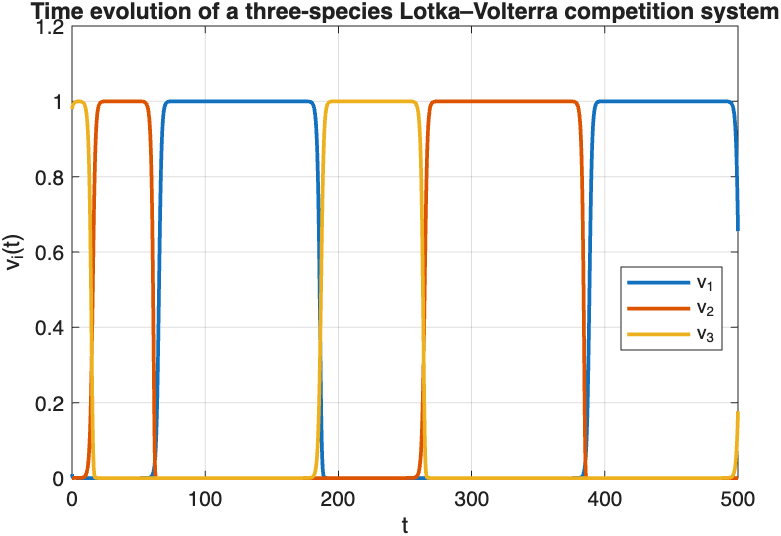}
\caption{Time evolution of the three-species Lotka--Volterra competition system over two successive patterns. In this simulation, $v_2(0)=10^{-6}$ and $v_2(T)=2.267\times 10^{-30}$. Parameters are given in Appendix \ref{sec:parameters}.}
\label{fig:two-patterns}
\end{figure}
Figure~\ref{fig:two-patterns} illustrates the dynamics of the three-species Lotka--Volterra competition system over two successive patterns. In this simulation, the second component decreases dramatically after the first full pattern: starting from $u_2(0)=10^{-6}$, we obtain $u_2(T)=2.267\times 10^{-30}$. At the same time, the solution still exhibits
the expected cyclic mechanism: during each pattern, the components become
successively dominant, and in particular species~2 recovers to an order-one
level before becoming extremely small again at the return time.

Moreover, the figure clearly shows that the second pattern takes much longer to complete than the first one. This observation is consistent with the upper and lower estimates obtained in Theorem~\ref{thm:main} for the pattern duration in terms of the initial value of $u_2$: the smaller the value of $u_2$ at the beginning of a pattern, the longer it takes for the solution to complete one full pattern. Thus, the numerical experiment supports the qualitative picture and the scaling behavior predicted by our theoretical analysis.

To illustrate the repeated-visit property stated in
Proposition~\ref{thm:corner-visits}, we plot in
Figure~\ref{fig:phase-trajectory} the phase trajectory of the solution in the
original variables. The trajectory successively approaches neighborhoods of the three boundary equilibria $\mathcal E_3=(0,0,1)$, $\mathcal E_2=(0,1,0)$, and $\mathcal E_1=(1,0,0)$, following the cyclic order
\begin{align*}
\mathcal E_3\longrightarrow \mathcal E_2
\longrightarrow \mathcal E_1\longrightarrow \mathcal E_3.
\end{align*}
This is consistent with Proposition~\ref{thm:corner-visits}, which asserts that, for each prescribed neighborhood size and for all sufficiently large patterns, the solution enters neighborhoods of these three equilibria in this order.
\begin{figure}[H]
\centering
\includegraphics[width=0.8\linewidth]{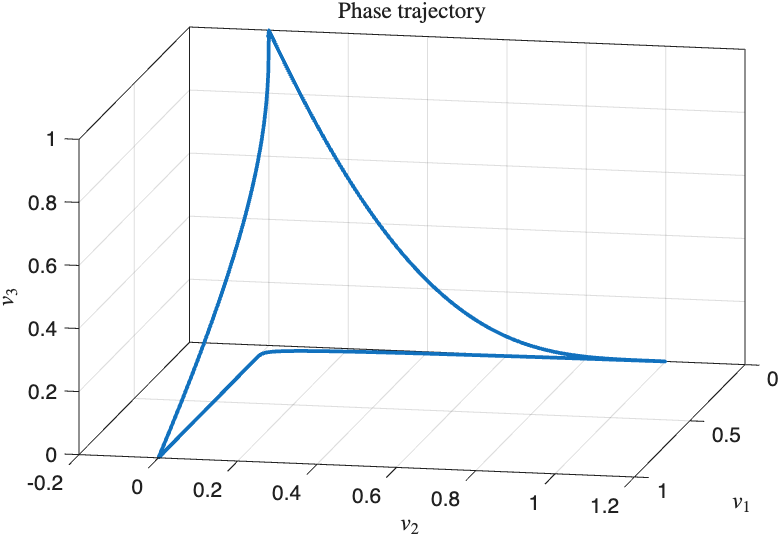}
\caption{Phase trajectory of system~\eqref{eq:main-1} in the original variables. The curve illustrates the cyclic dominance mechanism: the trajectory successively approaches neighborhoods of $(0,0,1)$, $(0,1,0)$, and $(1,0,0)$. Parameters are given in Appendix \ref{sec:parameters}.}
\label{fig:phase-trajectory}
\end{figure}

\section{Preliminaries}
First, we record some basic properties of solutions of system~\eqref{eq:system}.
We begin with positivity and boundedness, and then describe the pairwise
rock--paper--scissors structure of the original system~\eqref{eq:main-1}.

Throughout the paper, we consider solutions of \eqref{eq:system} with positive
initial data. Since each equation has the form
\begin{align*}
u_i'(t)=u_i(t)F_i(u(t)),
\end{align*}
the positive orthant is invariant. Hence, if $u_i(0)>0$ for $i=1,2,3$, then
$u_i(t)>0$ for all $t\ge0$ and all $i=1,2,3$. We shall also use the following
boundedness property.

\begin{proposition}\label{prop:dissipative}
Let $u(t)=\big(u_1(t),u_2(t),u_3(t)\big)$ be a solution of
\eqref{eq:system} with positive initial data. Then
\begin{align*}
\limsup_{t\to+\infty}u_1(t)\le\frac1\varepsilon,
\qquad
\limsup_{t\to+\infty}u_2(t)\le1,
\qquad
\limsup_{t\to+\infty}u_3(t)\le1.
\end{align*}
In particular, the solution is bounded on $\mathbb R_+$.
\end{proposition}
The proof is immediate by comparing each component with a logistic equation, so we omit the details.

Next, we point out the rock--paper--scissors relationship between the three
species in the original model~\eqref{eq:main-1} under
assumption~\eqref{eq:cond-sys-0}. If we restrict the dynamics to the invariant
plane $v_3\equiv0$, we obtain the two-species competitive Lotka--Volterra
system
\begin{align*}
\begin{cases}
v_1'(t)=v_1(t)\bigl(1-v_1(t)-k_{12}v_2(t)\bigr),\\
v_2'(t)=v_2(t)\bigl(1-k_{21}v_1(t)-v_2(t)\bigr).
\end{cases}
\end{align*}
Since $0<k_{12}<1<k_{21}=1/\varepsilon$, the classical theory for
two-species competitive Lotka--Volterra systems implies that, for positive
initial data,
\begin{align*}
(v_1(t),v_2(t))\longrightarrow(1,0)
\qquad\text{as }t\to+\infty;
\end{align*}
see, for instance, \cite[Chapter~3]{M02}. Thus species $1$ outcompetes
species $2$ in their pairwise interaction.

Similarly, on the invariant plane $v_1\equiv0$, the pair $(v_2,v_3)$ satisfies
a two-species competitive Lotka--Volterra system. Since $0<k_{23}<1<k_{32}$,
the same classical result shows that species $2$ outcompetes species $3$.
Finally, on the invariant plane $v_2\equiv0$, the inequalities
$0<k_{31}<1<k_{13}$ imply that species $3$ outcompetes species $1$.
Consequently, under assumption~\eqref{eq:cond-sys-0}, the full three-species
system~\eqref{eq:main-1} has the non-transitive cyclic competition structure
\begin{align*}
1\ \text{beats}\ 2,\qquad
2\ \text{beats}\ 3,\qquad
3\ \text{beats}\ 1.
\end{align*}
This is precisely the rock--paper--scissors configuration that motivates our
analysis.
\section{Proof of main theorem}
\subsection{Proof of Theorem \ref{thm:main}}
In this section we derive a sequence of quantitative estimates for the solution $(u_1,u_2,u_3)$ of \eqref{eq:system}.
Our arguments repeatedly rely on differential inequalities of the type
\begin{align*}
\rho_1 \le \frac{u_i'(t)}{u_i(t)} \le \rho_2 \quad\text{for } t\in[\tau_1,\tau_2],
\end{align*}
or logistic-type bounds
\begin{align*}
(\kappa_1-u_i(t))u_i(t)\le u_i'(t)\le (\kappa_2-u_i(t))u_i(t)\quad\text{for } t\in[\tau_1,\tau_2],
\end{align*}
with $0\le \tau_1<\tau_2$ and $0<\kappa_1<\kappa_2$.
Integrating over $[\tau_1,\tau_2]$ yields explicit two-sided estimates on $u_i(\tau_2)$ in terms of $u_i(\tau_1)$ and the length $\tau_2-\tau_1$.
In particular, knowing $u_i(\tau_1)$ and $u_i(\tau_2)$ provides bounds on $\tau_2-\tau_1$, and conversely, controlling $\tau_2-\tau_1$ yields upper/lower bounds on $u_i(\tau_2)$.

Throughout the paper, we use positivity and boundedness of solutions to obtain rough bounds, and then refine them on suitable time intervals.
We assume $\delta_1^\pm,\delta_2,\varepsilon\in(0,\tfrac14)$ and consider the initial data \eqref{eq:IC}.

%We introduce a sequence of stopping times $T_1<\cdots<T_3$ marking the successive phases of the cyclic dominance mechanism.
%We prove that each $T_i$ is finite (the finiteness proofs are collected in Appendix),
%and we show that at time $T_3$ the state satisfies $u_3(T_3)$ near $1$ and $u_2(T_3)$ is small, reproducing the initial configuration. This completes one full pattern and yields the dynamical picture stated in Theorem~\ref{thm:main}.

We now present the key lemmas, emphasizing the main mechanisms and omitting standard computations when no confusion may arise.

We decompose one pattern into three successive phases $0<T_1<T_2<T_3=T$, which we describe now. 
\begin{itemize}
\item \textbf{Phase 1: $u_3$ dominates, $u_1$ decreases, $u_2$ increases.} (from $0$ to ${T}_1$). At the beginning of Phase 1, $u_1(0)=1$, $u_2(0)$ is very small and $u_3(0)$ is close to the equilibrium value $1$. Hence the species $u_1$ loses the competition with $u_3$ and decreases sharply.  This releases the competitive pressure on $u_2$, which starts growing. The species $u_3$ stays close to the equilibrium value until $u_2$ reaches a sufficiently high value to win the competition and impact the density of $u_3$, at which point the next phase starts. 
\item \textbf{Phase 2: $u_2$ dominates, $u_3$ decreases, $u_1$ increases.} (from ${T}_1$ to ${T}_2$). At the beginning of Phase 2,  $u_1$ is extremely small, $u_2$ is reasonably high (sufficiently to drive $u_3$ to lower values) and $u_3$ is decaying. The decay of $u_3$ releases the pressure on $u_1$, which starts growing. $u_2$ remains at high density until $u_1$ reaches a sufficiently high threshold, to impact the density of $u_2$, at which point the next phase starts. 
\item \textbf{Phase 3: $u_1$ dominates, $u_2$ decreases, $u_3$ increases.} (from $T_{2}$ to $T_3$). During Phase 3, $u_1$ reaches very high values and $u_2$ decreases sharply; $u_3$ is decaying at first, but the release of the competitive pressure leads it to start growing again. When $u_3$ becomes sufficiently large, $u_1$ starts decreasing until it returns to its original value arbitrarily chosen to $u_1=1$. At the end of Phase 3, we are back to a situation close to the initial time, where $u_1=1$, $u_2$ is very small and $u_3$ is close to $1$; except $u_2$ is quantitatively much smaller than initially. 
\end{itemize}

{For a visual illustration of the three-phase decomposition $0<T_1<T_2<T_3=T$, we refer to Figure \ref{Figure3}, which displays the evolution of the three components on one full pattern and highlights the transition times $T_{1}$ and $T_{2}$ (dashed lines).}

\begin{figure}[H]
\centering
\includegraphics[width=\linewidth]{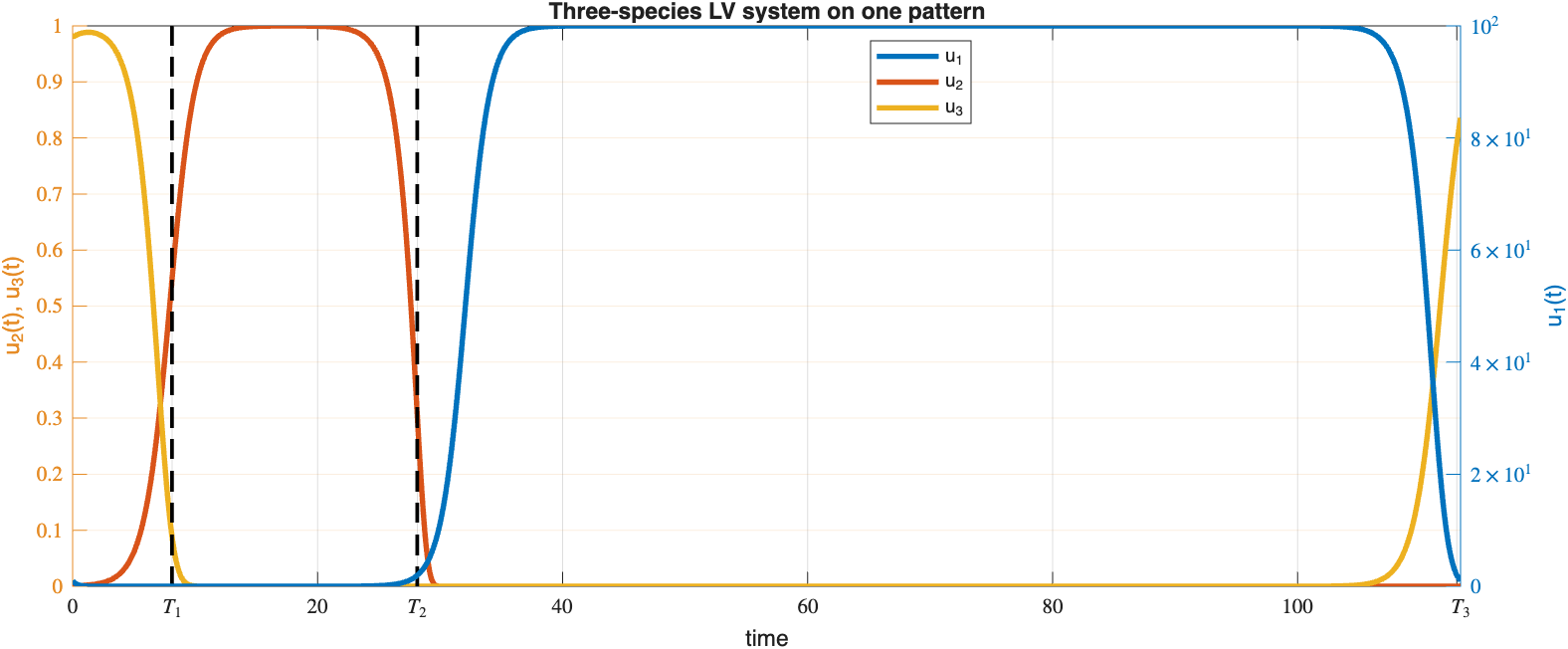}
\caption{Time evolution of the three-species Lotka--Volterra system over one pattern. The interval $[0,T]$ is divided into three successive phases by the times $0<T_1<T_2<T_3=T$. The dashed lines indicate the phase boundaries $T_1$ and $T_2$, while $T_3$ coincides with the end of the pattern ($T_3=T$). Parameters are given in Appendix \ref{sec:parameters}.}
\label{Figure3}
\end{figure}

We now introduce a sequence of lemmas and outline their proofs, focusing on the core ideas and qualitative mechanisms, and leaving aside the more standard technical computations.

%%%%%%%%%%%%%%%%%%%%%%%%%%%%%%%%%%%%%%%%%%%%%%%%%%%%%%%%%%%%%%%%%%%%%%%%%%%%%%%%%%%%%%%%%%%%%%%%%%%%%%%%%%%
\subsubsection{Phase 1: $u_3$ dominates, $u_1$ decreases, $u_2$ increases}
By our assumption, at the initial state, $u_3$ is close to its equilibrium and its strong competitor $u_2$ has a very low density; $u_1$ is decaying because it loses the competition with $u_3$. This dynamics continues until $u_2$, the strong competitor of $u_1$, reaches a sufficiently high density that impacts the dynamics of $u_3$. Thus we consider that phase 1 lasts until $u_2$ reaches the value
\begin{equation}\label{eq:nu3}
	\nu_3:=\frac{1}{2}\left(1-k_{23}+\frac{1}{k_{32}}\right). 
\end{equation}
We define $T_1$ as the ending time of phase 1; since the definition involves some intermediate times, we will define it later, in Lemma \ref{prop:main-1}.
%%%%%%%%%%%%%%%%%%%%%%%%%%%%%%%%%%%%%%%%%%%%%%%%%%%%%%%%%%%%%%%%%%%%%%%%%%%%%%%%%%%%%%%%%%%%%%%%%%%%%%%%%%%
\begin{lemma}\label{lem:main-11}
Define
\begin{align*}
\nu_1\coloneqq\frac{1-k_{23}}{2},\qquad
T_{11}\coloneqq\sup\Bigl\{t\ge 0\colon\ u_1(s)\ge \nu_1,\ \forall s\in[0,t]\Bigr\}.
\qquad
\end{align*}
Then we have  $T_{11}\lesssim 1$. Moreover, on $ [0, T_{11}]$, the following estimate holds true: 
\begin{align}\label{eq:u3-lower-T11}
\delta_1^{-}\lesssim u_2\lesssim \delta_1^{+}, \quad \frac12\left(1+\frac1{k_{13}}\right)<u_3<1.
\end{align}
\end{lemma}
\begin{remark}
More precisely, we have the estimate $T_{11}\leq T_{11}^{+}\coloneqq\frac{2}{k_{13}-1}\ln\left(\frac{1}{\nu_{1}}\right)$.
\end{remark}
%%%%%%%%%%%%%%%%%%%%%%%%%%%%%%%%%%%%%%%%%%%%%%%%%%%%%%%%%%%%%%%%%%%%%%%%%%%%%%%%%%%%%%%%%%%%%%%%%%%%%%%%%%%
\begin{proof}[Proof of Lemma \ref{lem:main-11}]

%\underline{\textbf{Step 1.}} $T_{11}<T_{11}^+=\frac{2}{k_{13}-1}\ln\left(\frac{1}{\nu_{1}}\right)$ and \eqref{eq:u3-lower-T11}.
	
On $[0,T_{11}^+]$, since $\frac{u_1'}{u_1}<1$ and $\frac{u_2'}{u_2}<1$, we have
	\begin{align}\label{eq:260224a}
u_1(t)<u_1(0)\,e^t\le e^{T_{11}^+},
\qquad
u_2(t)<u_2(0)\,e^t<\delta_1^{+}e^{T_{11}^+}.
\end{align}
Hence, on $[0,T_{11}^+]$
\begin{align*}
u_3'(t)>u_3(t)\Bigl(1-e^{T_{11}^+}(k_{31}\varepsilon+k_{32}\delta_1^{+})-u_3(t)\Bigr).
\end{align*}
By hypothesis, the parameters $\varepsilon \geq 0$, $\delta_{1}^{+}>0$ and $\delta_{2}>0$ are chosen small enough so that
\begin{align}\label{eq:est-smallness-T1}
1-e^{T_{11}^+}(k_{31}\varepsilon+k_{32}\delta_1^{+})>\frac12\left(1+\frac1{k_{13}}\right),
\end{align}
and
\begin{align*}
u_3(0)>1-\delta_2>\frac12\left(1+\frac1{k_{13}}\right),
\end{align*}
so comparison with the logistic ODE gives the estimate on $u_3$ in \eqref{eq:u3-lower-T11} for all  $t\in [0,T_{11}^+]$.

Using \eqref{eq:u3-lower-T11} in the first equation, we obtain on $[0,T_{11}^+]$,
\begin{align*}
\frac{u_1'}{u_1}<-\frac{k_{13}-1}{2}.
\end{align*}
Therefore $u_1(t)\le e^{-(k_{13}-1)t/2}$. Evaluating at $t=T_{11}^+$ gives
\begin{align*}
u_1(T_{11}^+)\le e^{-(k_{13}-1)T_{11}^+/2}=\nu_1.
\end{align*}
Hence the set defining $T_{11}$ is bounded above by $T_{11}^+$, so $T_{11}<T_{11}^+$.  
\medskip

%\underline{\textbf{Step 2.}} 
    For  $t\in [0,T_{11}]$, since $0<u_{1},u_2,u_3<1$, we have
\begin{align*}
\frac{u_2^{\prime}}{u_2}>-(1+k_{23}).
\end{align*}
Integrating on $[0,t]\subset[0,T_{11}]$ yields
\begin{align*}
u_2(t)\ge u_2(0)e^{-(1+k_{23})t}>\delta_1^{-}e^{-(1+k_{23})T_{11}^{+}}.
\end{align*}
	Together with the upper bound \eqref{eq:260224a} above, this proves the estimate for $u_2$ in \eqref{eq:u3-lower-T11}, and ends the proof of Lemma \ref{lem:main-11}.
\end{proof}

%%%%%%%%%%%%%%%%%%%%%%%%%%%%%%%%%%%%%%%%%%%%%%%%%%%%%%%%%%%%%%%%%%%%%%%%%%%%%%%%%%%%%%%%%%%5
%%%%%%%%%%%%%%%%%%%%%%%%%%%%%%%%%%%%%%%%%%%%%%%%%%%%%%%%%%%%%%%%%%%%%%%%%%%%%%%%%%%%%%%%%%%5

%%%%%%%%%%%%%%%%%%%%%%%%%%%%%%%%%%%%%%%%%%%%%%%%%%%%%%%%%%%%%%%%%%%%%%%%%%%%%%%%%%%%%%%%%%%%%%%%%%%%%%%%%%%

\begin{lemma}\label{lem:main-12}
Define
\begin{align*}
\nu_{2}\coloneqq\min\left\{\frac{k_{13}-1}{4k_{13}k_{32}},\frac{\nu_{1}}{2}\right\},
\qquad
T_{12}\coloneqq\sup\Bigl\{t\ge T_{11}:\ u_2(s)\le\nu_2,\ \forall s\in[T_{11},t]\Bigr\}.
\end{align*}
Then it holds that $-\ln\delta_1^+\lesssim  T_{12}-T_{11}\lesssim -\ln\delta_{1}^{-}$ and, for all $t\in[T_{11},T_{12}]$,
\begin{align}\label{eq:est-u3-T12}
u_1(t)<\nu_1\quad \text{ and } \quad  u_3(t)>\frac12\Bigl(1+\frac1{k_{13}}\Bigr).
\end{align}
Moreover, we have 
\begin{align}\label{eq:est-u1-T12}
\ln\delta_{1}^{-}\lesssim \ln u_1(T_{12})\lesssim \ln\delta_{1}^{+}.
\end{align}
\end{lemma}
%%%%%%%%%%%%%%%%%%%%%%%%%%%%%%%%%%%%%%%%%%%%%%%%%%%%%%%%%%%%%%%%%%%%%%%%%%%%%%%%%%%%%%%%%%%%%%%%%%%%%%%%%%%
\begin{proof}
    We divide this proof in several steps.

    \textbf{\underline{Step 1.} Proof of \eqref{eq:est-u3-T12}.}

    \noindent 
    By definition of $T_{12}$,  for all $t\in [T_{11}, T_{12}]$, the following inequalities hold true. 
\begin{equation}\label{app:eq:u2-two-upper}
u_2(t)\le \frac{k_{13}-1}{4k_{13}k_{32}}
\qquad\text{and}\qquad
u_2(t)\le \frac{\nu_1}{2}=\frac{1-k_{23}}{4}
.
\end{equation}
     Let us prove that  $u_1(t)<\nu_1$ for all $ t\in( T_{11}, T_{12}]$.
    Suppose by contradiction that there exists a time $ t\in( T_{11}, T_{12}]$ with   $u_1(t)\geq \nu_1$, and define
\begin{align*}
t_0:=\inf\{t>T_{11}\colon\ u_1(t)=\nu_1\}.
\end{align*}
    Recall \eqref{eq:u3-lower-T11} from Lemma \ref{lem:main-11}; in particular we have that $u_1'(T_{11})<0$.
    Therefore $t_0\in(T_{11},T_{12})$, $u_1(t)<\nu_1$ for all $t\in(T_{11},t_0)$, and necessarily
there exists an interval $(a,b)\subset(T_{11},t_0)$ on which $u_1'(t)>0$.
On $(a,b)$,
\begin{align*}
0<\frac{u_1'(t)}{u_1(t)} \le 1-k_{13}u_3(t),
\end{align*}
hence $u_3(t)<\frac{1}{k_{13}}$ on $(a,b)$.
On the other hand, by Lemma~\ref{lem:main-11},
    \begin{align}\label{eq:260224c}
u_3(T_{11})>\frac12\Bigl(1+\frac1{k_{13}}\Bigr)=:m>\frac1{k_{13}}.
\end{align}
Therefore $u_3$ must cross the level $m$ on $[T_{11},b]$. Define the first hitting time
\begin{align*}
\tau:=\inf\{t\in[T_{11},b]:\ u_3(t)\le m\}.
\end{align*}
Then $\tau\in(T_{11},b]$, $u_3(\tau)=m$, and $u_3^{\prime}(\tau)\le 0$.
Evaluating the third equation at $t=\tau$ gives
\begin{align*}
0\ge \frac{u_3'(\tau)}{u_3(\tau)}
=1-\varepsilon k_{31}u_1(\tau)-k_{32}u_2(\tau)-m,
\end{align*}
hence
\begin{align*}
k_{32}u_2(\tau)\ge \frac{k_{13}-1}{2k_{13}}-\varepsilon k_{31}u_1(\tau).
\end{align*}
Since $\tau<t_0$ we have $u_1(\tau)<\nu_1$.
Choosing $\varepsilon$ small enough so that
%\begin{align*}
$\varepsilon k_{31}\nu_1<\frac{k_{13}-1}{4k_{13}},$
%\end{align*}
we obtain
%\begin{align*}
$u_2(\tau)>\frac{k_{13}-1}{4k_{13}k_{32}},$
%\end{align*}
    which contradicts \eqref{app:eq:u2-two-upper}. The contradiction  proves our claim that $u_1(t)<\nu_1$ for all $ t\in( T_{11}, T_{12}]$. 
    \medskip

    Now suppose by contradiction that $u_3(t)$ crosses the level $m$ for some $t\in [T_{11}, T_{12}]$, then by reproducing the end of the previous argument, we similarly obtain a contradiction with \eqref{app:eq:u2-two-upper}. We omit the details. This ends the proof of Step 1.  
\medskip

    \noindent \textbf{\underline{Step 2.} Finiteness and estimate of $T_{12}$.} 

\noindent Since  $u_2(t)\le \nu_2\le \nu_1/2$ and $0<u_3(t)<1$, it follows that
\begin{align*}
\frac{\nu_1}{2}<\frac{u_2'(t)}{u_2(t)}<1,
\qquad t\in[T_{11},T_{12}].
\end{align*}
Integrating over $[T_{11},T_{12}]$ and using $u_2(T_{12})=\nu_2$ together with
Lemma~\ref{lem:main-11} (which gives $C_1\delta_1^-<u_2(T_{11})<C_2\delta_1^+$)
yields two-sided bounds of the form
\begin{align}\label{eq:est-T11-T12}
\ln\left(\frac{\nu_{2}}{C_{2}\,\delta_{1}^{+}}\right)
<T_{12}-T_{11}
<\frac{2}{\nu_{1}}\,
\ln\left(\frac{\nu_{2}}{C_{1}\,\delta_{1}^{-}}\right),
\end{align}
which ends Step 2.
\medskip

    \noindent \textbf{\underline{Step 3.} Proof of \eqref{eq:est-u1-T12}.}

\noindent By \eqref{eq:est-u3-T12} we have on $[T_{11},T_{12}]$
\begin{align*}
-(k_{12}+k_{13})\le \frac{u_1'(t)}{u_1(t)}\le -\frac{k_{13}-1}{2}.
\end{align*}
Integrating and using $u_1(T_{11})=\nu_1$ gives
\begin{align*}
\ln\nu_1-(k_{12}+k_{13})(T_{12}-T_{11})<\ln u_1(T_{12})<\ln\nu_1-\frac{k_{13}-1}{2}(T_{12}-T_{11}).
\end{align*}
We combine this with \eqref{eq:est-T11-T12} to obtain the following  \eqref{eq:est-u1-T12}. This  completes the proof of Lemma \ref{lem:main-12}.
\end{proof}

%%%%%%%%%%%%%%%%%%%%%%%%%%%%%%%%%%%%%%%%%%%%%%%%%%%%%%%%%%%%%%%%%%%%%%%%%%%%%%%%%%%%%%%%%%%%%%%%%%%%%%%%%%%
\begin{lemma}\label{prop:main-1}
Define
$\nu_{3}\coloneqq \frac{1}{2}\left(1-k_{23}+\dfrac{1}{k_{32}}\right),$
and
    \begin{align}\label{eq:defT1}
T_{1}\coloneqq \sup\bigl\{t\ge T_{12}\colon\ u_{2}(s)\le\nu_{3},\ \forall\,s\in[T_{12},t]\bigr\}.
\end{align}
Then $T_1$ is finite and satisfies $T_1-T_{12}\lesssim 1 $. Moreover, there exists a constant $\eta\in(0,1/2)$, depending only on the coefficients $k_{ij}$, such that
\begin{align}\label{eq:est-u1-T1}
%\alpha_{1}^{-}\ln\delta_{1}^{-}+\beta_{2}^{-}
\ln\delta_{1}^{-}
\lesssim \ln u_{1}(T_1) \lesssim
%\alpha_{1}^{+}\ln\delta_{1}^{+}+\beta_{2}^{+}.
\ln\delta_{1}^{+} \quad\  \text{and}\ \quad 
\eta\le u_3(T_{1})\le 1-\eta.
\end{align}
\end{lemma}
\begin{remark}
    More precisely, we have the estimate $T_1-T_{12}\leq T_1^{+}:=\dfrac{2}{2\nu_{1}+\nu_{3}}
\ln\left(
\dfrac{\tfrac{\nu_{1}}{\nu_{2}}+\tfrac{\nu_{3}}{2\nu_{2}}-1}
{\tfrac{\nu_{1}}{\nu_{3}}-\tfrac{1}{2}}
\right)$.
\end{remark}
%%%%%%%%%%%%%%%%%%%%%%%%%%%%%%%%%%%%%%%%%%%%%%%%%%%%%%%%%%%%%%%%%%%%%%%%%%%%%%%%%%%%%%%%%%%%%%%%%%%%%%%%%%%
\begin{proof}[Proof of Lemma \ref{prop:main-1}]
From Lemma~\ref{lem:main-12}, there exist constants $\alpha_1^+,\beta_1^+>0$ such that
\begin{align*}
u_{1}(T_{12})<\beta_1^+\,(\delta_{1}^{+})^{\alpha_{1}^{+}}.
\end{align*}
Since $\frac{u_1'}{u_1}<1$, it follows that, for every
$t\in[T_{12},T_{12}+T_1^{+}]$,
\begin{align*}
u_{1}(t)
<
u_{1}(T_{12})\,e^{t-T_{12}}
<
\beta_1^+\,(\delta_{1}^{+})^{\alpha_{1}^{+}}\,e^{T_1^{+}}.
\end{align*}
Choosing $\delta_{1}^{+}$ sufficiently small gives
\begin{align*}
u_{1}(t)<\nu_{1}-\frac{\nu_{3}}{2},
\qquad
t\in[T_{12},T_{12}+T_1^{+}].
\end{align*}
Using $u_3(t)<1$ and $u_1(t)<\nu_1-\nu_3/2$ in the $u_2$-equation yields, on
$[T_{12},T_{12}+T_1^{+}]$,
\begin{align*}
\frac{u_2'}{u_2}
=
1-u_1-u_2-k_{23}u_3
>
\nu_1+\frac{\nu_3}{2}-u_2.
\end{align*}
Hence, by comparison with the solution $v(t)$ of the logistic ODE
\begin{align*}
v'=v\left(\nu_{1}+\frac{\nu_{3}}{2}-v\right),
\qquad
v(T_{12})=\nu_{2},
\end{align*}
we obtain $u_2(t)>v(t)$ for all $t\in[T_{12},T_{12}+T_1^{+}]$. By the definition
of $T_1^{+}$, the explicit solution satisfies $v(T_{12}+T_1^{+})=\nu_{3}$, hence
$u_{2}(T_{12}+T_1^{+})>\nu_{3}$. Therefore $T_1<T_{12}+T_1^{+}$ and, in particular,
$T_1<\infty$.

On $[T_{12},T_1]$ we have
\begin{align*}
-(k_{12}+k_{13})<\frac{u_1'}{u_1}<1,
\end{align*}
therefore
\begin{align*}
-(k_{12}+k_{13})(T_1-T_{12})
<
\ln\left(\frac{u_1(T_1)}{u_1(T_{12})}\right)
<
T_1-T_{12}.
\end{align*}
Since $T_1-T_{12}\le T_1^{+}$ is bounded above by a positive constant and recalling
\eqref{eq:est-u1-T12} from Lemma~\ref{lem:main-12}, we obtain the estimate on $u_1$
in \eqref{eq:est-u1-T1}.

We now estimate $u_3(T_1)$ from above. Since $\frac{u_1'}{u_1}<1$, we have, for every $t\in[T_{12},T_1]$,
\begin{align*}
u_1(t)
<
u_1(T_{12})e^{t-T_{12}}
<
\beta_1^+(\delta_1^+)^{\alpha_1^+}e^{T_1^+}.
\end{align*}
Hence, up to choosing $\delta_1^+$ sufficiently small,
\begin{align*}
u_1(t)<\frac14\left(1-k_{23}-\frac1{k_{32}}\right),
\qquad
t\in[T_{12},T_1].
\end{align*}
Using again $u_3(t)<1$ and
\begin{align*}
u_2(t)\le\nu_3
=\frac12\left(1-k_{23}+\frac1{k_{32}}\right),
\qquad
t\in[T_{12},T_1],
\end{align*}
we obtain
\begin{align*}
\frac{u_2'}{u_2}>\frac14\left(1-k_{23}-\frac1{k_{32}}\right)>0,
\end{align*}
for every $t\in[T_{12},T_1]$. Hence $u_2$ is increasing on $[T_{12},T_1]$. Since
\begin{align*}
u_2(T_{12})=\nu_2<\frac{1-k_{23}}4,
\qquad
u_2(T_1)=\nu_3=\frac12\left(1-k_{23}+\frac1{k_{32}}\right),
\end{align*}
and $u_2$ is continuous, there exists $\hat\tau_{2}\in(T_{12},T_1)$ such that
\begin{align*}
u_2(\hat\tau_{2})=\frac14\left(1-k_{23}+\frac3{k_{32}}\right).
\end{align*}
Since $u_2$ is increasing on $[T_{12},T_1]$, it follows that
\begin{align*}
u_2(t)>\frac14\left(1-k_{23}+\frac3{k_{32}}\right),
\qquad t\in[\hat\tau_{2},T_1].
\end{align*}
Therefore, for every $t\in[\hat\tau_{2},T_1]$,
\begin{align*}
\frac{u_3'}{u_3}<
-\frac14\big(\left(1-k_{23}\right)k_{32}-1\big)=:-r_3<0.
\end{align*}
Hence $u_3$ is decreasing on $[\hat\tau_{2},T_1]$, and
\begin{align*}
u_3(t)<u_3(\hat\tau_{2})\ e^{-r_3(t-\hat\tau_{2})}\le e^{-r_3(t-\hat\tau_{2})},
\qquad t\in[\hat\tau_{2},T_1],
\end{align*}
where we used $u_3(\hat\tau_{2})\le1$. Moreover, since $\frac{u_2'}{u_2}<1$, we have
\begin{align*}
u_2(T_1)<u_2(\hat\tau_{2})\ e^{T_1-\hat\tau_{2}},
\end{align*}
and therefore
\begin{align*}
e^{T_1-\hat\tau_{2}}>
\frac{u_2(T_1)}{u_2(\hat\tau_{2})}=
\frac{\frac12\left(1-k_{23}+\frac1{k_{32}}\right)}
{\frac14\left(1-k_{23}+\frac3{k_{32}}\right)}=\frac{2-2k_{23}+\frac2{k_{32}}}{1-k_{23}+\frac3{k_{32}}}.
\end{align*}
Evaluating the previous upper bound at $t=T_1$, we obtain
\begin{align*}
u_3(T_1)<e^{-r_3(T_1-\hat\tau_{2})}
<\left(\frac{1-k_{23}+\frac3{k_{32}}}{2-2k_{23}+\frac2{k_{32}}}\right)^{r_3}.
\end{align*}
To prove the lower bound for $u_3(T_1)$, we remark that
$\frac{u_3'}{u_3}>-(k_{31}+k_{32})$, and therefore
\begin{align}\label{eq:est-u3-T1}
u_3(T_1)
>
u_3(T_{12})\,\exp\bigl(-(k_{31}+k_{32})(T_1-T_{12})\bigr).
\end{align}
Combining \eqref{eq:est-u3-T12} and \eqref{eq:est-u3-T1}, we obtain
\begin{align*}
u_{3}(T_1)>
\dfrac{k_{13}+1}{2k_{13}}\,
\left(\dfrac{\tfrac{\nu_{1}}{\nu_{3}}-\tfrac{1}{2}}{\tfrac{\nu_{1}}{\nu_{2}}+\tfrac{\nu_{3}}{2\nu_{2}}-1}
\right)^{\tfrac{2(k_{31}+k_{32})}{2\nu_{1}+\nu_{3}}}
.
\end{align*}
This ends the proof of Lemma~\ref{prop:main-1}.
\end{proof}
%%%%%%%%%%%%%%%%%%%%%%%%%%%%%%%%%%%%%%%%%%%%%%%%%%%%%%%%%%%%%%%%%%%%%%%%%%%%%%%%%%%%%%%%%%%%%%%%%%%%%%%%%%%

%%%%%%%%%%%%%%%%%%%%%%%%%%%%%%%%%%%%%%%%%%%%%%%%%%%%%%%%%%%%%%%%%%%%%%%%%%%%%%%%%%%%%%%%%%%%%%%%%%%%%%%%%%%
\subsubsection{Phase 2: $u_2$ dominates, $u_3$ decreases, $u_1$ increases}
At the end of Phase 1, $u_2$ has reached a sufficiently high value, so $u_3$ has to decrease because it cannot sustain the competition. This diminishes the amount of pressure on $u_1$, that will quickly start rising again. We declare the end of Phase 2 at $T_2$, when  $u_1$ has  reached the threshold $u_1(T_2)=2$ (see the actual definition below, in Lemma \ref{prop:main-2}).
%%%%%%%%%%%%%%%%%%%%%%%%%%%%%%%%%%%%%%%%%%%%%%%%%%%%%%%%%%%%%%%%%%%%%%%%%%%%%%%%%%%%%%%%%%%%%%%%%%%%%%%%%%%
\begin{lemma}\label{lem:main-21}
Let
\begin{align*}
\nu_{4}&\coloneqq\dfrac{1}{2}\left(1-k_{23}-\frac{1}{k_{32}}\right),
\qquad T_{21}\coloneqq\sup\Bigl\{t\ge T_1\colon u_{1}(s)\le\nu_{4},\ \forall s\in[T_1,t]\Bigr\}.
\end{align*}
Then $-\ln\delta_{1}^{+}\lesssim T_{21}-T_1\lesssim - \ln\delta_{1}^{-}$, and we have 
\begin{equation}\label{eq:app-u2-ge-nu3}
u_2(t)\geq \nu_3  \text{ for all }t\in[T_1, T_{21}], 
\end{equation}
and
\begin{align}\label{eq:est-u3-T4}
\ln\delta_{1}^{-}\lesssim\ln u_{3}(T_{21})\lesssim\ln\delta_{1}^{+}.
\end{align}
\end{lemma}
%%%%%%%%%%%%%%%%%%%%%%%%%%%%%%%%%%%%%%%%%%%%%%%%%%%%%%%%%%%%%%%%%%%%%%%%%%%%%%%%%%%%%%%%%%%%%%%%%%%%%%%%%%%
\begin{remark}
    Note that $\nu_{4}>0$ follows from our assumption  $k_{32}(1-k_{23})>1$ in \eqref{eq:cond-sys-2}.
\end{remark}
%%%%%%%%%%%%%%%%%%%%%%%%%%%%%%%%%%%%%%%%%%%%%%%%%%%%%%%%%%%%%%%%%%%%%%%%%%%%%%%%%%%%%%%%%%%%%%%%%%%%%%%%%%%
\begin{proof}[Proof of Lemma \ref{lem:main-21}]
    
    %\noindent \textbf{\underline{Step 1.} } 

    \noindent We first prove \eqref{eq:app-u2-ge-nu3}. For all $t\in[T_1, T_{21}] $ we have, by definition of $T_{21}$, $u_1(t)\leq \nu_4$.  Since moreover we know that $0<u_{3}(t)<1$, we obtain
\begin{align*}
\frac{u_{2}'(t)}{u_{2}(t)}\ge (1-\nu_{4}-k_{23})-u_{2}(t)
=\nu_{3}-u_{2}(t).
\end{align*}
and \eqref{eq:app-u2-ge-nu3} follows from a simple comparison argument.
\medskip

Next we consider  \eqref{eq:est-u3-T4} and the estimate on $T_{21}-T_1$.
 Set 
 $\gamma \coloneqq k_{32}\nu_{3}-1=\frac12\bigl(k_{32}(1-k_{23})-1\bigr)>0$ (the sign follows from \eqref{eq:cond-sys-2}) 
 and
$\overline{\gamma}\coloneqq k_{31}\nu_{4}+k_{32}>0$.
Using \eqref{eq:app-u2-ge-nu3}, $u_{1}\le\nu_{4}$ and $0<u_{3}<1$, we obtain on $[T_1,T_{21}]$ that
\begin{align*}
-\overline{\gamma}
<\frac{u_{3}'}{u_{3}}<1-k_{32}\nu_{3}
=-\gamma.
\end{align*}
Integrating gives
\begin{equation}\label{eq:u3-int-short}
\ln u_{3}(T_1)-\overline{\gamma}(T_{21}-T_1)
<\ln u_{3}(T_{21})
<\ln u_{3}(T_1)-\gamma(T_{21}-T_1).
\end{equation}
Let
%\begin{align*}
$\nu_{3}^{\star}\coloneqq\frac{1-k_{12}}{2k_{13}}$
%\end{align*}
and let us show that there exists $T_1^\star\geq T_1$ with $T_1^\star-T_1\lesssim 1$ such that 
 $u_{3}(T_1^\star)\leq \nu_{3}^{\star}$. Indeed, if $u_{3}(T_1)> \nu_{3}^{\star}$
then we know from $\frac{u_{3}'}{u_{3}}<-\gamma$ that $u_{3}$ is strictly decreasing on $[T_1,T_{21}]$.
Moreover, since $\frac{u_{1}^{\prime}}{u_{1}}<1$ on $[T_1,T_{21}]$ and $u_{1}(T_{21})=\nu_{4}$, we have
\begin{align*}
T_{21}-T_1>\ln\nu_{4}-\ln u_{1}(T_1)
>\alpha_{1}^{+}\ln\left(\frac{1}{\delta_{1}^{+}}\right)+\ln\nu_{4},
\end{align*}
where $0<\alpha_1^+\lesssim 1$ is a constant given by  \eqref{eq:est-u1-T1}. Plugging this into
\begin{align*}
u_{3}(T_{21})\le u_{3}(T_1)e^{-\gamma(T_{21}-T_1)},
\end{align*}
shows that for $\delta_{1}^{+}$ small enough
 one has $u_{3}(T_{21})<\nu_{3}^{\star}$.
Hence, in any case, there exists  $T_1^{\star}\in[T_1,T_{21})$ such that
$u_{3}(T_1^{\star})\le\nu_{3}^{\star}$, and $T_1^{\star}-T_1\lesssim 1$.

Since $u_{3}$ is decreasing on $[T_1,T_{21}]$, we have $u_{3}(t)<\nu_{3}^{\star}$ for all $t\in (T_1^\star, T_{21}].$
Therefore, we obtain
\begin{align*}
1>\frac{u_{1}'}{u_{1}}>1-\varepsilon\nu_{4}-k_{12}-k_{13}\nu_{3}^{\star}
=\frac{1-k_{12}}{2}-\varepsilon\nu_{4}.
\end{align*}
Choosing $\varepsilon<\frac{1-k_{12}}{4\nu_{4}}$ we get $1>\frac{u_{1}^{\prime}}{u_{1}}>\frac{1-k_{12}}{4}>0$ on $[T_1^\star,T_{21}]$,
so $u_{1}$ is increasing and, recalling $u_{1}(T_{21})=\nu_{4}$,
\begin{align*}
    -\ln\delta_1^+\lesssim 
\ln\left(\frac{\nu_{4}}{u_{1}(T_1^\star )}\right)
<T_{21}-T_1^\star
<\frac{4}{1-k_{12}}\ln\left(\frac{\nu_{4}}{u_{1}(T_1^\star)}\right)
    \lesssim -\ln \delta_1^-.
\end{align*}
Then  \eqref{eq:u3-int-short}  yields the desired estimate for $\ln u_{3}(T_{21})$. This ends the proof of Lemma \ref{lem:main-21}.
%with constants $\alpha_{2}^{\pm}>0$, $\beta_{3}^{\pm}\in\mathbb{R}$ depending only on $(k_{ij})$.
\end{proof}
%%%%%%%%%%%%%%%%%%%%%%%%%%%%%%%%%%%%%%%%%%%%%%%%%%%%%%%%%%%%%%%%%%%%%%%%%%%%%%%%%%%%%%%%%%%%%%%%%%%%%%%%%%%
\begin{lemma}\label{prop:main-2}
Let
\begin{align*}
T_2=\sup\left\{t\ge T_{21}\colon\ u_{1}(s)\le2,\ \forall s\in[T_{21},t]\right\}.
\end{align*}
Then $T_2-T_{21}\lesssim 1$. Moreover, there exists a constant
$\eta\in(0,1/2)$, depending only on the coefficients $k_{ij}$, such that
\begin{align}\label{eq:est-u3-T5}
\eta\le u_2(T_{2})\le1-\eta \quad \text{ and }\quad \ln\delta_{1}^{-}\lesssim \ln u_{3}(T_2) \lesssim \ln\delta_{1}^{+}.
\end{align}
\end{lemma}
%%%%%%%%%%%%%%%%%%%%%%%%%%%%%%%%%%%%%%%%%%%%%%%%%%%%%%%%%%%%%%%%%%%%%%%%%%%%%%%%%%%%%%%%%%%%%%%%%%%%%%%%%%%
\begin{proof}
Set
\begin{align*}
T_2^{+}\coloneqq \frac{4}{1-k_{12}}\ln\left(\frac{4}{\nu_{4}}\right).
\end{align*}
We assume $\varepsilon>0$ and $\delta_{1}^{+}>0$ are chosen small enough so that
\begin{equation}\label{eq:smallness-step5}
\varepsilon\,\nu_{4}e^{T_2^+}<\frac{1-k_{12}}{4},
\qquad
(\delta_{1}^{+})^{\alpha_{2}^{+}}e^{T_2^+}<\frac{1-k_{12}}{2k_{13}},
\end{equation}
where $\alpha_2^{+}$ is a constant such that
\begin{align*}
\ln(u_3(T_{21}))\le \alpha_2^+ \ln \delta_1^+
\end{align*}
in \eqref{eq:est-u3-T4} of Lemma~\ref{lem:main-21}.

For any $t\in[T_{21},T_{21}+T_2^+]$, using $\frac{u_1'}{u_1}<1$ and
$\frac{u_3'}{u_3}<1$, we have
\begin{align*}
u_{1}(t)\le u_{1}(T_{21})\,e^{t-T_{21}}
=\nu_{4}e^{t-T_{21}}
\le \nu_{4}e^{T_2^+},
\end{align*}
and
\begin{align*}
u_{3}(t)<u_{3}(T_{21})\,e^{t-T_{21}}
<(\delta_{1}^{+})^{\alpha_{2}^{+}}e^{T_2^+}.
\end{align*}
Hence \eqref{eq:smallness-step5} implies
\begin{align*}
\varepsilon u_{1}(t)<\frac{1-k_{12}}{4},
\qquad
k_{13}u_{3}(t)<\frac{1-k_{12}}{2},
\qquad t\in[T_{21},T_{21}+T_2^+].
\end{align*}
Since $0<u_2<1$, we deduce on $[T_{21},T_{21}+T_2^+]$ that
\begin{align*}
\frac{u_{1}'}{u_{1}}
>
1-\frac{1-k_{12}}{4}-k_{12}-\frac{1-k_{12}}{2}
=
\frac{1-k_{12}}{4}.
\end{align*}
Integrating on $[T_{21},T_{21}+T_2^+]$ gives
\begin{align*}
u_{1}(T_{21}+T_2^+)
>
u_{1}(T_{21})\exp\left(\frac{1-k_{12}}{4}(T_{21}+T_2^+-T_{21})\right)
=
\nu_{4}\exp\left(\ln\frac{4}{\nu_{4}}\right)
=
4>2.
\end{align*}
Therefore $T_2<T_{21}+T_2^+$, and in particular
\begin{equation}\label{eq:step5-time}
0<T_2-T_{21}\le T_2^+\lesssim 1.
\end{equation}
On $[T_{21},T_2]$, since
\begin{align*}
-(k_{31}+k_{32})<\frac{u_{3}^{\prime}}{u_{3}}<1,
\end{align*}
we obtain from \eqref{eq:step5-time} that
\begin{align*}
\ln u_{3}(T_{21})-(k_{31}+k_{32})\,T_2^+
<
\ln u_{3}(T_2)
<
\ln u_{3}(T_{21})+T_2^+.
\end{align*}
Combining this with Lemma~\ref{lem:main-21},
we obtain the estimate on $u_3$ in \eqref{eq:est-u3-T5}.

We now prove the upper bound for $u_2(T_2)$. Recall that
\begin{align*}
u_1(T_{21})=\nu_4=\frac12\left(1-k_{23}-\frac1{k_{32}}\right)<\frac12,
\qquad
u_1(T_2)=2.
\end{align*}
Since $u_1$ is continuous and increasing on $[T_{21},T_2]$, there exists
$\hat\tau_{1}\in(T_{21},T_2)$ such that
\begin{align*}
u_1(\hat\tau_{1})=\frac32.
\end{align*}
Hence
\begin{align*}
u_1(t)>\frac32,
\qquad
t\in[\hat\tau_{1},T_2].
\end{align*}
Therefore, for every $t\in[\hat\tau_{1},T_2]$,
\begin{align*}
\frac{u_2'}{u_2}<-\frac12.
\end{align*}
It follows that $u_2$ is decreasing on $[\hat\tau_{1},T_2]$, and
\begin{align*}
u_2(t)<u_2(\hat\tau_{1})\ e^{-\frac12(t-\hat\tau_{1})}\le e^{-\frac12(t-\hat\tau_{1})},
\qquad
t\in[\hat\tau_{1},T_2],
\end{align*}
where we used $u_2(\hat\tau_{1})\le1$. We next estimate $T_2-\hat\tau_{1}$. Since $\frac{u_1'}{u_1}<1$, we have
\begin{align*}
u_1(T_2)<u_1(\hat\tau_{1})\ e^{T_2-\hat{\tau}_{1}},
\end{align*}
and therefore
\begin{align*}
e^{T_2-\hat\tau}>\frac{u_1(T_2)}{u_1(\hat\tau_{1})}
=\frac43.
\end{align*}
Combining this with the previous estimate at $t=T_2$, we obtain
\begin{align*}
u_2(T_2)<e^{-\frac12(T_2-\hat\tau_{1})}<\left(\frac34\right)^{1/2}.
\end{align*}
To prove the lower bound for $u_2(T_2)$, we remark that for all
$t\in[T_{21},T_2]$ one has $u_{1}(t)\le 2$ and $0<u_{2}(t),u_{3}(t)<1$, hence
\begin{align*}
\frac{u_{2}'}{u_{2}}
=
1-u_{1}-u_{2}-k_{23}u_{3}
>
-(2+k_{23}).
\end{align*}
Integrating on $[T_{21},T_2]$, using $u_{2}(T_{21})\ge \nu_{3}$ from
\eqref{eq:app-u2-ge-nu3} and $T_2-T_{21}\le T_2^+$, we get
\begin{align*}
u_{2}(T_2)
>
u_{2}(T_{21})\exp\bigl(-(2+k_{23})(T_2-T_{21})\bigr)
\ge
\nu_{3}\exp\bigl(-(2+k_{23})T_2^+\bigr).
\end{align*}
Hence
\begin{align}\label{eq:u2-T5-lower}
u_{2}(T_{2})>\nu_{3}\exp\bigl(-(2+k_{23})T_2^+\bigr).
\end{align}
This proves \eqref{eq:est-u3-T5} and completes the proof of
Lemma~\ref{prop:main-2}.
\end{proof}
%%%%%%%%%%%%%%%%%%%%%%%%%%%%%%%%%%%%%%%%%%%%%%%%%%%%%%%%%%%%%%%%%%%%%%%%%%%%%%%%%%%%%%%%%%%%%%%%%%%%%%%%%%%
\subsubsection{Phase 3: $u_1$ dominates, $u_2$ decreases, $u_3$ increases}
At the end of Phase 2, $u_1$ has reached a sufficiently high value, so $u_2$ has started to decrease. 
In turn, the vanishing competitive pressure on $u_3$ leads to an increase of the population. Phase 3  ends at $T_3$, when $u_1$ has  returned to its original value $u_1(T_3)=1$, thus completing the pattern.

During this phase, $u_2$ has to sustain a huge amount of competition from $u_1$, which is typically comparable to a power of $\frac{1}{\varepsilon}$. Thus it reaches very low levels, magnitudes below the initial data $u_2(0)$.
%%%%%%%%%%%%%%%%%%%%%%%%%%%%%%%%%%%%%%%%%%%%%%%%%%%%%%%%%%%%%%%%%%%%%%%%%%%%%%%%%%%%%%%%%%%%%%%%%%%%%%%%%%%
\begin{lemma}\label{lem:main-31}
Let
\begin{align*}
T_{31}\coloneqq\sup\Bigl\{t\ge T_2\colon\ u_{1}(s)\le\tfrac{1-k_{12}}{4\varepsilon},\forall s\in[T_2,t]\Bigr\}.
\end{align*}
Then $-\ln\varepsilon\lesssim T_{31}-T_2\lesssim-\ln\varepsilon$, and
\begin{align}\label{eq:est-u2-T31}
(1+\tfrac{1}{\varepsilon})\ln\varepsilon\lesssim\ln u_{2}(T_{31})\lesssim\ln\varepsilon.
\end{align}
Moreover, there exists $0<\lambda_{1}^{-},\lambda_{1}^{+}\lesssim 1$ such that
\begin{align}\label{eq:est-u3-T31}
\ln\delta_{1}^{-}-\lambda_{1}^{-}\ln\varepsilon
\lesssim \ln u_{3}(T_{31})\lesssim \ln\delta_{1}^{+}-\lambda_{1}^{+}\ln\varepsilon.
\end{align}
\end{lemma}

\begin{proof}
\noindent \textbf{\underline{Step 1.} Finiteness and estimate of $T_{31}$. } 

Set $T_{31}^{+}\coloneqq \frac{2}{1-k_{12}}\ln\Bigl(\frac{1-k_{12}}{4\varepsilon}\Bigr) $.
We assume that $\delta_{1}^{+}>0$ and $\varepsilon$ are chosen small enough so that
\begin{equation}\label{eq:smallness-step31}
\left(\delta_{1}^{+}\right)^{\alpha_{3}^{+}}\left(\frac{1-k_{12}}{4\varepsilon}\right)^{\tfrac{2}{1-k_{12}}}<\min\left\{\frac{1-k_{31}}{4},\frac{1-k_{12}}{2k_{13}}\right\},
\end{equation}
and
\begin{align*}
\dfrac{1-k_{12}}{4\varepsilon}>2.
\end{align*}
where $\alpha_{3}^{+}$ is a constant such that $\ln u_3(T_{2})\leq \alpha_{3}^+ \ln\delta_1^+$ in \eqref{eq:est-u3-T5} of Lemma~\ref{prop:main-2}.

For any $t\in[T_{2},T_{2}+T_{31}^+]$, since $\frac{u_{3}'}{u_3}<1$, we have
$u_{3}(t)\le u_{3}(T_{2})\ e^{t-T_{2}}\le \left(\delta_{1}^{+}\right)^{\alpha_{3}^{+}}e^{t-T_{2}}\le \left(\delta_{1}^{+}\right)^{\alpha_{3}^{+}}e^{T_{31}^+}$. Hence, by \eqref{eq:smallness-step31},
\begin{align*}
u_{3}(t)<\frac{1-k_{31}}{4},
\qquad
k_{13}u_{3}(t)<\frac{1-k_{12}}{2},\qquad t\in[T_{2},T_{2}+T_{31}^+].
\end{align*}
Since $0<u_2<1$, we deduce on $[T_{2},T_{2}+T_{31}^+]$ that 
$\frac{u_{1}'}{u_{1}}>1-k_{12}-\frac{1-k_{12}}{2}-\varepsilon u_{1}=\frac{1-k_{12}}{2}-\varepsilon u_{1}$. By a comparison argument, for $t\in[T_{2},T_{2}+T_{31}^+]$,
\begin{align*}
u_{1}>\dfrac{1-k_{12}}{2\varepsilon+\left(\tfrac{1-k_{12}}{2}-2\varepsilon\right)\exp\left(-(\tfrac{1-k_{12}}{2})(t-T_{2})\right)}.
\end{align*}
Evaluating at $t=T_{2}+T_{31}^{+}$ gives
\begin{align*}
u_{1}(T_{2}+T_{31}^+)>\dfrac{1-k_{12}}{4\varepsilon}.
\end{align*}
Therefore $T_{31}<T_{2}+T_{31}^+$, and in particular,
\begin{equation}\label{eq:step31-time-upper}
0<T_{31}-T_{2}\le T_{31}^+\lesssim-\ln\varepsilon.
\end{equation}
To obtain a lower bound for $T_{31}-T_{2}$, using $\tfrac{u_{1}^{\prime}}{u_{1}}<1$, $u_{1}(T_{2})=2$, and $u_{1}(T_{31})=\tfrac{1-k_{12}}{4\varepsilon}$, we get $-\ln\varepsilon\lesssim T_{31}-T_{2}$.

\noindent \textbf{\underline{Step 2.} Estimates on $u_2$ \eqref{eq:est-u2-T31} and $u_3$ \eqref{eq:est-u3-T31}.}

As a consequence, on $[T_{2},T_{31}]$ we have $\tfrac{u_{1}^{\prime}}{u_{1}}>\tfrac{1-k_{12}}{4}>0$, hence $u_{1}>u_{1}(T_{2})=2$. Therefore,
\begin{align*}
-\left(1+\tfrac{1}{\varepsilon}\right)\lesssim\dfrac{u_{2}^{\prime}}{u_{2}}\lesssim -1.
\end{align*}
Recalling that $1\lesssim u_2(T_{2})<1$, we obtain
\begin{align*}
(1+\tfrac{1}{\varepsilon})\ln\varepsilon\lesssim\ln u_{2}(T_{31})\lesssim\ln\varepsilon,
\end{align*}
which proves \eqref{eq:est-u2-T31}.

Let $\nu_{5}^{\star}\coloneqq\frac{1-k_{31}}{2k_{32}}$. We claim that there exists $T_{31}^\star\geq T_{2}$ with $T_{31}^\star-T_{2}\lesssim 1$ such that $u_{2}(T_{31}^\star)\leq\nu_{5}^{\star}$. Indeed, if $u_{2}(T_{2})>\nu_{5}^{\star}$, then
$\ln u_{2}(T_{31})\lesssim\ln\varepsilon$ implies that, for $\varepsilon$ small enough, $u_{2}(T_{31})<\nu_{5}^{\star}$. Hence, in any case, there exists $T_{31}^{\star}\in[T_{2},T_{31})$ such that
$u_{2}(T_{31}^{\star})\le\nu_{5}^{\star}$. Moreover, $T_{31}^{\star}-T_2\lesssim 1$.

On $[T_{31}^{\star},T_{31}]$, we have $0<\varepsilon u_{1}<1$, $0<u_{2}<\nu_{5}^{\star}$ and $0<u_{3}<\tfrac{1-k_{31}}{4}$. Therefore, we obtain
\begin{align*}
1>\frac{u_{3}'}{u_{3}}>1-k_{31}-k_{32}\cdot\frac{1-k_{31}}{2k_{32}}-\frac{1-k_{31}}{4}=\frac{1-k_{31}}{4}.
\end{align*}
Combining this with $\ln\delta_{1}^{-}\lesssim \ln u_{3}(T_2) \lesssim \ln\delta_{1}^{+}$ of Lemma \ref{prop:main-2}, we get
\begin{align*}
\ln\delta_{1}^{-}-\dfrac{1-k_{31}}{4}\,\ln\varepsilon\lesssim \ln u_{3}(T_{31})\lesssim \ln\delta_{1}^{+}-\ln\varepsilon.
\end{align*}
This proves \eqref{eq:est-u3-T31} and ends the proof of Lemma \ref{lem:main-31}.
\end{proof}

\begin{lemma}\label{lem:main-32}
Let
\begin{align*}
\nu_{6}\coloneqq\min\left\{\frac{1-k_{31}}{4},\frac{1-k_{12}}{2k_{13}}\right\} \text{ and }
T_{32}\coloneqq\sup\Bigl\{t\ge T_{31}\colon\ u_{3}(s)\le\nu_{6}\ \text{for all}\ s\in[T_{31},t]\Bigr\}.
\end{align*}
Then $-\ln \delta_1^{+}+\lambda_{2}\ln\varepsilon\lesssim T_{32}-T_{31}\lesssim -\ln\delta_1^{-}+\lambda_{2}\ln\varepsilon$,  we get   
\begin{align}\label{eq:est-u1-T32}
\dfrac{1}{\varepsilon}
\lesssim u_{1}(T_{32})<\dfrac{1}{\varepsilon}.
\end{align}
Moreover, there exist $\lambda_{2}^{-},\lambda_{2}^{+}>0$ such that
\begin{align}\label{eq:est-u2-T32}
(1+\tfrac{1}{\varepsilon})(\ln\delta_{1}^{-}-\lambda_{2}^{-}\ln\varepsilon)
\lesssim \ln u_{2}(T_{32})
\lesssim (1+\tfrac{1}{\varepsilon})(\ln \delta_{1}^{+}-\lambda_{2}^{+}\ln\varepsilon).
\end{align}
\end{lemma}

\begin{proof}
%%%%%%%%%%%%%%%%%%%%%%%%%%%%%%%%%%%%%%%%%%%%%%%%%%%%%%%%%%%%%%%%%%%%%%%%%%%%%%%%%%%%%%%%%%%%%%%%%%%%%%%%%%%
\noindent \textbf{\underline{Step 1.} Finiteness and estimate of $T_{32}$. } 

\noindent For all $t\in [T_{31},T_{32}]$, we have 
\begin{equation}\label{eq:u3-small-contrad-app}
u_{3}(t)\le \nu_{6}\le\frac{1-k_{12}}{2k_{13}}
\quad\text{and}\quad
u_{3}(t)\le \nu_{6}\le \frac{1-k_{31}}{4}.
\end{equation}
Using $0<u_{2}(t)<1$ and \eqref{eq:u3-small-contrad-app}, the first equation of \eqref{eq:system} yields
$\frac{u_{1}^{\prime}(t)}{u_{1}(t)}>\frac{1-k_{12}}{2}-\varepsilon u_{1}(t)$ for all $t\in [T_{31}, T_{32}]$. By a simple comparison argument, we deduce that 
\begin{equation}\label{eq:u1-ge-2}
u_1(t)\ge \dfrac{1-k_{12}}{4\varepsilon}>2,\qquad t\in[T_{31},T_{32}].
\end{equation}
This proves \eqref{eq:est-u1-T32}. Moreover, it implies that $\tfrac{u_{2}^{\prime}}{u_{2}}\le-1$ for all $t\in [T_{31},T_{32}]$. In particular, $u_2$ is strictly decreasing on $[T_{31},T_{32}]$. Since $u_2(T_{31})<\nu_{5}^{\star}$, it follows that $u_{2}<\nu_{5}^{\star}$ on $[T_{31},T_{32}]$. Using this, the uniform bound $\varepsilon u_{1}(t)\le 1$, and the third equation of \eqref{eq:system},
we obtain for all $t\in[T_{31},T_{32}]$,
\begin{align*}
\frac{u_{3}'}{u_{3}}>1-k_{31}-k_{32}\cdot\frac{1-k_{31}}{2k_{32}}-\frac{1-k_{31}}{4}=\frac{1-k_{31}}{4}>0,\ \text{so}\ u_3(t)\geq u_3(T_{31})\,e^{\frac{1-k_{31}}{4}(t-T_{31})}.
\end{align*}
Thus $T_{32}-T_{31}\lesssim -\ln\delta_1^{-}+\lambda_{1}^{+}\ln\varepsilon$. From $\frac{u_3'}{u_3}\leq 1$ we get $u_3(t)\leq u_3(T_{31})\,e^{t-T_{31}}$, so $-\ln \delta_1^++\lambda_{1}^{-}\ln\varepsilon\lesssim  T_{32}-T_{31}$.
\medskip

\noindent \textbf{\underline{Step 2.} Estimates on $u_2$ \eqref{eq:est-u2-T32}.} 

From \eqref{eq:est-u1-T32}, we have
\begin{align}\label{est:est-u2-T32-pre}
-\left(1+\dfrac{1}{\varepsilon}\right)\lesssim \dfrac{u_{2}^{\prime}}{u_{2}}\lesssim -\left(1+\dfrac{1}{\varepsilon}\right).
\end{align}
Combining \eqref{est:est-u2-T32-pre}, \eqref{eq:est-u2-T31} and $-\ln \delta_1^++\lambda_{1}^{-}\ln\varepsilon\lesssim T_{32}-T_{31}\lesssim -\ln\delta_1^{-}+\lambda_{1}^{+}\ln\varepsilon$ gives us \eqref{eq:est-u2-T32} and ends the proof of Lemma \ref{lem:main-32}.
\end{proof}
%%%%%%%%%%%%%%%%%%%%%%%%%%%%%%%%%%%%%%%%%%%%%%%%%%%%%%%%%%%%%%%%%%%%%%%%%%%%%%%%%%%%%%%%%%%%%%%%%%%%%%%%%%%
\begin{lemma}\label{prop:main-33}
Let
\begin{align*}
\nu_{7}\coloneqq \frac12\left(1-k_{31}+\frac1{k_{13}}\right),
\qquad
T_{33}\coloneqq\sup\left\{t\ge T_{32}\colon\ u_{3}(s)\le\nu_{7},\ \forall s\in[T_{32},t]\right\}.
\end{align*}
Then $T_{33}-T_{32}\lesssim 1$. Moreover, there exists a constant
$\eta\in(0,1/2)$, depending only on the coefficients $k_{ij}$, such that  
\begin{align}\label{eq:est-u1-T33}
\eta\le \varepsilon u_{1}(T_{33})\le 1-\eta,
\end{align}
and there exist constant $\lambda_{3}^{-},\lambda_{3}^{+}>0$ such that
\begin{align}\label{eq:est-u2-T33}
(1+\tfrac{1}{\varepsilon})(\ln\delta_{1}^{-}-\lambda_{3}^{-}\ln\varepsilon)
\lesssim \ln u_{2}(T_{33})
\lesssim (1+\tfrac{1}{\varepsilon})(\ln \delta_{1}^{+}-\lambda_{3}^{+}\ln\varepsilon).
\end{align}
\end{lemma}
%%%%%%%%%%%%%%%%%%%%%%%%%%%%%%%%%%%%%%%%%%%%%%%%%%%%%%%%%%%%%%%%%%%%%%%%%%%%%%%%%%%%%%%%%%%%%%%%%%%%%%%%%%%
\begin{proof}
Recall from Lemma~\ref{lem:main-32} that
\begin{align*}
\ln u_2(T_{32})\lesssim \left(1+\frac{1}{\varepsilon}\right)
\left(\ln \delta_{1}^{+}-\lambda_{2}^{+}\ln\varepsilon\right),
\end{align*}
and let $0<\alpha_{4}^+\lesssim 1$ be such that
\begin{align*}
u_2(T_{32})
\le
\left(\frac{\delta_{1}^{+}}{\varepsilon^{\lambda_{2}^{+}}}\right)^{\alpha_{4}^{+}\left(1+\tfrac{1}{\varepsilon}\right)}.
\end{align*}
We choose $\delta_{1}^{0}>0$ sufficiently small so that
\begin{align*}
\left(\frac{\delta_{1}^{+}}{\varepsilon^{\lambda_{2}^{+}}}\right)^{\alpha_{4}^{+}\left(1+\tfrac{1}{\varepsilon}\right)}\zeta
<
\frac{1}{4k_{32}}\left(1-k_{31}-\frac{1}{k_{13}}\right),
\end{align*}
where
\begin{align*}
\zeta=
\left[
\frac{\frac12\left(1-k_{31}+\frac1{k_{13}}\right)}
{\frac14\left(1-k_{31}-\frac1{k_{13}}\right)}
\cdot
\left(
\frac{\frac34(1-k_{31})+\frac{1}{4k_{13}}}{\nu_{6}}-1
\right)
\right]^{\frac{4}{3(1-k_{31})+\frac1{k_{13}}}}
>1.
\end{align*}
Define
\begin{align}\label{eq:def-T7bar-rev}
T_{33}^{+}:=\ln\zeta.
\end{align}

For all $t\in[T_{32},T_{32}+T_{33}^{+}]$, we have
\begin{align*}
u_{2}(t)
<
u_{2}(T_{32})\,e^{T_{33}^{+}}
<
\left(\frac{\delta_{1}^{+}}{\varepsilon^{\lambda_{2}^{+}}}\right)^{\alpha_{4}^{+}\left(1+\tfrac{1}{\varepsilon}\right)}
\zeta
<
\frac{1}{4k_{32}}\left(1-k_{31}-\frac{1}{k_{13}}\right).
\end{align*}
Therefore, on $[T_{32},T_{32}+T_{33}^{+}]$,
\begin{align}\label{eq:u3prime-T33-rev}
u_{3}^{\prime}
>
u_{3}\left[\frac34(1-k_{31})+\frac{1}{4k_{13}}-u_{3}\right]
=:u_{3}(\kappa-u_{3}).
\end{align}
Hence, by comparison with the logistic ODE
\begin{align*}
v'=v(\kappa-v),\qquad v(T_{32})=\nu_{6},
\end{align*}
we obtain $u_{3}(t)>v(t)$ for all $t\in[T_{32},T_{32}+T_{33}^{+}]$.
By the definition of $T_{33}^{+}$, the explicit solution satisfies
\begin{align*}
v(T_{32}+T_{33}^{+})=\nu_{7},
\end{align*}
hence
\begin{align*}
u_{3}(T_{32}+T_{33}^{+})>\nu_{7}.
\end{align*}
It follows that
\begin{align*}
T_{33}-T_{32}<T_{33}^{+}=\ln\zeta\lesssim 1,
\end{align*}
and in particular $T_{33}<\infty$.

On the other hand, from $\frac{u_{3}^{\prime}}{u_{3}}<1$ we get
\begin{align*}
T_{33}-T_{32}>
\ln\left(\frac{u_{3}(T_{33})}{u_{3}(T_{32})}\right)
=
\ln\frac{\nu_{7}}{\nu_{6}}.
\end{align*}
Since
\begin{align*}
\nu_{7}
=
\frac12\left(1-k_{31}+\frac1{k_{13}}\right)
>
\frac{1-k_{31}}{2}
\ge
2\min\left\{\frac{1-k_{31}}{4},\frac{1-k_{12}}{2k_{13}}\right\}
=
2\nu_{6},
\end{align*}
we deduce
\begin{align}\label{eq:est-T32-T33-rev}
\ln2<T_{33}-T_{32}<\ln\zeta.
\end{align}

We next estimate $u_{1}(T_{33})$. Since
\begin{align*}
-(k_{12}+k_{13})<\frac{u_1'}{u_1}<1
\qquad\text{on }[T_{32},T_{33}],
\end{align*}
we obtain
\begin{align*}
u_1(T_{32})\,e^{-(k_{12}+k_{13})(t-T_{32})}
<
u_1(t)
<
u_1(T_{32})\,e^{t-T_{32}},
\qquad t\in[T_{32},T_{33}].
\end{align*}
Using \eqref{eq:est-u1-T32} from Lemma~\ref{lem:main-32} and
\eqref{eq:est-T32-T33-rev}, we infer that
\begin{align*}
\varepsilon u_{1}(T_{33})>\dfrac{1-k_{12}}{4}\cdot\left(\dfrac{1}{\zeta}\right)^{k_{12}+k_{13}}.
\end{align*}

To prove the upper bound, we first claim that $u_3$ is increasing on $[T_{32},T_{33}]$.
Indeed, from \eqref{eq:u3prime-T33-rev} and the fact that $u_3\le \nu_7$ on
$[T_{32},T_{33}]$, we have
\begin{align*}
\frac{u_{3}^{\prime}}{u_{3}}
>
\frac34(1-k_{31})+\frac{1}{4k_{13}}-\nu_7
=
\frac14\left(1-k_{31}-\frac1{k_{13}}\right)
>0.
\end{align*}
Hence $u_3$ is increasing on $[T_{32},T_{33}]$. Since
\begin{align*}
u_3(T_{32})=\nu_6\le \frac{1-k_{31}}{4},
\qquad
u_3(T_{33})=\nu_7=\frac12\left(1-k_{31}+\frac1{k_{13}}\right),
\end{align*}
and $u_3$ is continuous, there exists $\hat\tau_{3}\in(T_{32},T_{33})$ such that
\begin{align*}
u_{3}(\hat\tau_{3})
=
\frac14\left(1-k_{31}+\frac{3}{k_{13}}\right).
\end{align*}
Since $u_{3}$ is increasing on $[T_{32},T_{33}]$, it follows that
\begin{align*}
u_{3}(t)>
\frac14\left(1-k_{31}+\frac{3}{k_{13}}\right),
\qquad
t\in[\hat\tau_{3},T_{33}].
\end{align*}
Therefore, for every $t\in[\hat\tau_{3},T_{33}]$,
\begin{align*}
\frac{u_{1}^{\prime}}{u_{1}}<-\frac14\bigl((1-k_{31})k_{13}-1\bigr)=:-r_{1}<0.
\end{align*}
Hence $u_{1}$ is decreasing on $[\hat\tau_{3},T_{33}]$, and
\begin{align*}
u_{1}(t)<u_{1}(\hat\tau_{3})\,e^{-r_{1}(t-\hat\tau_{3})}
<\frac{1}{\varepsilon}\ e^{-r_{1}(t-\hat\tau_{3})},
\qquad
t\in[\hat\tau_{3},T_{33}],
\end{align*}
where we used $u_{1}(\hat\tau_{3})<\frac1\varepsilon$, since
$\hat\tau_3>T_{32}$ and $u_1<\frac1\varepsilon$ on $[T_{32},T_{33}]$. Moreover, since $\frac{u_{3}^{\prime}}{u_{3}}<1$, we have
\begin{align*}
u_{3}(T_{33})
<
u_{3}(\hat\tau_{3})\,e^{T_{33}-\hat\tau_{3}},
\end{align*}
and therefore
\begin{align*}
e^{T_{33}-\hat\tau_{3}}
>
\frac{u_{3}(T_{33})}{u_{3}(\hat\tau_{3})}
=
\frac{\frac12\left(1-k_{31}+\frac1{k_{13}}\right)}
{\frac14\left(1-k_{31}+\frac3{k_{13}}\right)}
=
\frac{2-2k_{31}+\frac2{k_{13}}}{1-k_{31}+\frac3{k_{13}}}.
\end{align*}
Evaluating the previous estimate for $u_1$ at $t=T_{33}$, we obtain
\begin{align*}
\varepsilon u_{1}(T_{33})<e^{-r_{1}(T_{33}-\hat\tau_{3})}
<\left(
\frac{1-k_{31}+\frac3{k_{13}}}{2-2k_{31}+\frac2{k_{13}}}
\right)^{r_1}
=:C_1^+,
\end{align*}
where $0<C_1^+<1$. This proves \eqref{eq:est-u1-T33}.

Finally, since
\begin{align}\label{eq:est-u2-T33-pre-rev}
-\left(1+\frac{1}{\varepsilon}\right)\lesssim\frac{u_{2}^{\prime}}{u_{2}}<0
\end{align}
on $[T_{32},T_{33}]$, combining \eqref{eq:est-T32-T33-rev} and
\eqref{eq:est-u2-T33-pre-rev} gives
\begin{align*}
-\left(1+\frac{1}{\varepsilon}\right)\lesssim
\ln u_{2}(T_{33})-\ln u_{2}(T_{32})<0.
\end{align*}
Together with \eqref{eq:est-u2-T32}, this yields \eqref{eq:est-u2-T33} and
completes the proof of Lemma~\ref{prop:main-33}.
\end{proof}
%%%%%%%%%%%%%%%%%%%%%%%%%%%%%%%%%%%%%%%%%%%%%%%%%%%%%%%%%%%%%%%%%%%%%%%%%%%%%%%%%%%%%%%%%%%%%%%%%%%%%%%%%%%
\begin{lemma}\label{prop:main-3}
Let
\begin{align*}
T_3
=\sup\left\{t\ge T_{33}\colon u_{1}(s)\ge1,\ \forall s\in[T_{33},t]\right\}.
\end{align*}
Then  $-\ln\varepsilon\lesssim T_3-T_{33}\lesssim-\ln\varepsilon$. Moreover, there exist constants $\lambda^{-},\lambda^{+}>0$ such that
\begin{align}\label{eq:est-u2-T3}
(1+\tfrac{1}{\varepsilon})(\ln\delta_{1}^{-}-\lambda^{-}\ln\varepsilon)
\lesssim \ln u_{2}(T_{3})
\lesssim (1+\tfrac{1}{\varepsilon})(\ln \delta_{1}^{+}-\lambda^{+}\ln\varepsilon).
\end{align}
and one has 
\begin{align}\label{eq:est-u3-T3}
\ln\big(1-u_{3}(T_3)\big)\lesssim \ln\varepsilon.
\end{align}
\end{lemma}
%%%%%%%%%%%%%%%%%%%%%%%%%%%%%%%%%%%%%%%%%%%%%%%%%%%%%%%%%%%%%%%%%%%%%%%%%%%%%%%%%%%%%%%%%%%%%%%%%%%%%%%%%%%
\begin{proof}
On $[T_{33},T_{3}]$, since $u_{1}(t)>1$ and $u_{2},u_{3}>0$, we have $\tfrac{u_{2}'}{u_{2}}<0$, and thus $u_{2}$ is strictly decreasing on $[T_{33},T_{3}]$. Thus the upper bound in \eqref{eq:est-u2-T3} holds true.  

Now we know that $u_2(t)\leq \left(\tfrac{\delta_{1}^{+}}{\varepsilon^{\lambda_{3}^{+}}}\right)^{\alpha_{4}^{+}\left(1+\tfrac{1}{\varepsilon}\right)}$  for  all $t\in [T_{33},T_3]$, where $\alpha_{4}^+$ is a constant $0<\alpha_{4}^+\lesssim 1$;  select $\delta_{1}^{0}>0$ sufficiently small to ensure  
\begin{align*}
\left(\tfrac{\delta_{1}^{+}}{\varepsilon^{\lambda_{3}^{+}}}\right)^{\alpha_{4}^{+}\left(1+\tfrac{1}{\varepsilon}\right)}
<\dfrac{1}{2k_{32}}\left(1-k_{31}-\dfrac{1}{k_{13}}\right).
\end{align*}
Then, on $[T_{33},T_3]$ we have
%\begin{align*}
$u_{2}(t)<\dfrac{1}{2k_{32}}\left(1-k_{31}-\dfrac{1}{k_{13}}\right)$,
%\end{align*}
which implies
%\begin{align*}
$u_{3}^{\prime}>u_{3}\left[\frac{1}{2}\left(1-k_{31}+\frac{1}{k_{13}}\right)-u_{3}\right]$.
%\end{align*}
Using $u_{3}(T_{33})=\dfrac{1}{2}\left(1-k_{31}+\dfrac{1}{k_{13}}\right)$ and the comparison principle, we obtain
\begin{align*}
u_{3}(t)>\dfrac{1}{2}\left(1-k_{31}+\dfrac{1}{k_{13}}\right),
\qquad t\in[T_{33},T_3].
\end{align*}
Consequently,
\begin{align*}
-\gamma^{-}=:-(k_{12}+k_{13})<\frac{u_{1}'}{u_{1}}<-\frac{1}{2}\big[(1-k_{31})k_{13}-1\big]:=-\gamma^{+}.
\end{align*}
Integrating from $T_{33}$ to $T_3$ and using \eqref{eq:est-u1-T33} and $u_{1}(T_3)=1$  we obtain 
\begin{align}\label{eq:T8T7-raw}
-\ln\varepsilon\lesssim \ln\left[\big(u_{1}(T_{33})\big)^{1/\gamma^{-}}\right]
<T_3-T_{33}<\ln\left[\big(u_{1}(T_{33})\big)^{1/\gamma^{+}}\right]\lesssim-\ln\varepsilon. 
\end{align}
This provides the desired estimate on $T_{3}-T_{33}$.
\medskip

Next we estimate $u_2$ from below. Since $u_1(t)<\tfrac{1}{\varepsilon}$ we have 
\begin{align*}
-\left(1+\dfrac{1}{\varepsilon}\right)\lesssim\frac{u_{2}^{\prime}}{u_{2}},
\end{align*}
on $[T_{33},T_3]$, and thus
$-(1+\tfrac{1}{\varepsilon})\ln\varepsilon\lesssim\ln u_{2}(T_3)-\ln u_{2}(T_{33})$. From this and \eqref{eq:est-u2-T33}, we deduce
the lower bound in \eqref{eq:est-u2-T3}.
\medskip 

Finally, we estimate $u_3$. Recall that $\dfrac{1}{\varepsilon}\lesssim u_{1}(T_{33})$, then by selecting $\varepsilon_{0}$ small enough we have $u_{1}(T_{33})>\tfrac{1}{\sqrt{\varepsilon}}$. Combining this with the monotonicity of $u_{1}$ on $[T_{33},T_{3}]$, one can find $T_{33}^{\star}\in(T_{33},T_{3})$ such that $u_{1}(T_{33}^{\star})=\tfrac{1}{\sqrt{\varepsilon}}$. Then, by a similar argument, $e^{T_{3}-T_{33}^{\star}}>\left(\tfrac{1}{\varepsilon}\right)^{\tfrac{1}{2\gamma^{-}}}$. 

Recall that $u_2(T_{33}^{\star})\leq \left(\tfrac{\delta_{1}^{+}}{\varepsilon^{\lambda_{2}^{+}}}\right)^{\alpha_{4}^{+}\left(1+\tfrac{1}{\varepsilon}\right)}$ and let us set 
\begin{align*}
\nu
:=\Bigl(\tfrac{1}{\nu_{7}}-1\Bigr)^{2}\,\varepsilon^{\tfrac{1}{2\gamma^{-}}}.
\end{align*}
Then up to reducing $\delta_1^0$ and  $\varepsilon_0$ we may impose
\begin{align*}
k_{31}\varepsilon\,u_{1}(t)
<k_{31}\varepsilon\,u_{1}(T_{33}^{\star})<\frac{\nu}{2},\quad\ \text{and}\ \quad 
k_{32}\,u_{2}(t)
<k_{32}\,u_{2}(T_{33}^{\star})<\frac{\nu}{2},
\end{align*}
for all $t\in[T_{33}^{\star},T_3]$.
Hence,
%\begin{align*}
$u_{3}'(t) >u_{3}(t)\Bigl(\big(1-\nu\big)-u_{3}(t)\Bigr)$,
%\end{align*}
and by comparing with the logistic ODE we obtain, 
\begin{align*}
u_{3}(t)>
\frac{1-\nu}{1+\Bigl(\frac{1-\nu}{\nu_{7}}-1\Bigr)e^{-(1-\nu)(t-T_{33}^{\star})}}.
\end{align*}
Since $0<\nu<\frac12$, we have $1-\nu>\frac12$, and
$\frac{1-\nu}{\nu_{7}}-1<\frac{1}{\nu_{7}}-1\eqqcolon\eta$.
Therefore,
\begin{align*}
u_{3}(t)
&>\frac{1-\nu}{1+\eta\,e^{-(1-\nu)(t-T_{33}^{\star})}}
>\frac{1-\nu}{1+\eta\,e^{-(t-T_{33}^{\star})/2}},
\qquad t\in[T_{33}^{\star},T_3].
\end{align*}
Evaluating at $t=T_3$, we obtain for small $\delta_1^+\ll 1$,
\begin{align*}
u_{3}(T_3)
&>\frac{1-\nu}{1+\eta\,e^{-(T_3-T_{33}^{\star})/2}}
>\frac{1-\nu}{1+\mathcal{O}\left(\varepsilon^{\tfrac{1}{4\gamma^{-}}}\right)}.
\end{align*}
Since $\nu=\mathcal{O}\left(\varepsilon^{\tfrac{1}{2\gamma^{-}}}\right)$,
we obtain
\begin{align*}
u_{3}(T_3)>1-\mathcal{O}\left(\varepsilon^{\tfrac{1}{4\gamma^{-}}}\right),
\end{align*}
which is exactly \eqref{eq:est-u3-T3}. This completes the proof of Lemma \ref{prop:main-3}.
\end{proof}

\subsection{Proof of Theorem \ref{thm:main}}
\begin{proof}[Proof of Theorem \ref{thm:main}]
The main results have already been established by Lemma \ref{lem:main-11}--\ref{prop:main-3}. The finiteness of $T$ and the first estimate in \eqref{eq:est-That-T} are a consequence of these Lemmas. Indeed from Lemma \ref{lem:main-11}, Lemma \ref{lem:main-12} and Lemma \ref{prop:main-1}, we get the estimate $-\ln \delta_1^{+}\lesssim T_1\lesssim -\ln\delta_1^{-}$; similarly, Lemma \ref{lem:main-21} and \ref{prop:main-2} give $-\ln \delta_1^{+}\lesssim T_2-T_1\lesssim -\ln\delta_1^{-}$ and Lemma \ref{lem:main-31}, \ref{lem:main-32}, \ref{prop:main-33} and \ref{prop:main-3} give $-\ln\delta_{1}^{+}+\kappa\ln\varepsilon\lesssim T_3-T_2\lesssim -\ln\delta_1^{-}-\ln\varepsilon$ for $\kappa>0$; from which we deduce
\begin{equation*}
-\ln\delta_{1}^{+}+\kappa\ln\varepsilon\lesssim T=T_3\lesssim -\ln\delta_1^{-}-\ln\varepsilon.
\end{equation*}
The estimate \eqref{eq:est-u-main} comes from Lemma \ref{prop:main-3}. Finally, the estimate \eqref{eq:est-u2} comes from Lemma \ref{prop:main-1} (more precisely, that $u_{2}(T_1)=\nu_3$, with  $\widehat{T}:=T_1$); the second part of \eqref{eq:est-That-T} follows from \eqref{eq:est-u1-T33} in Lemma \ref{prop:main-33}. This ends the proof of Theorem \ref{thm:main}.
\end{proof}
\subsection{Proof of Proposition \ref{thm:corner-visits}}
\begin{lemma}\label{lem:small-u1}
Suppose that
\begin{align*}
0<c_0\le u_2(0)<1,\qquad u_3(0)\in(0,1).
\end{align*}
Let $\delta\in(0,1)$. Then there exist constants
\begin{align*}
\nu_\delta>0,
\qquad
\hat T_\delta>0,
\end{align*}
depending only on $\delta$, $c_0$, and the coefficients $(k_{ij})$, such
that the following holds: for every $T_\delta>\hat T_\delta$, if
\begin{align*}
0<u_1(t)\le \nu_\delta
\qquad\text{for all }t\in[0,T_\delta],
\end{align*}
then
\begin{align*}
1-\delta\le u_2(t)<1,
\qquad
0<u_3(t)\le \delta
\qquad\text{for all }t\in[\hat T_\delta,T_\delta].
\end{align*}
In particular,
\begin{align*}
\|u(t)-(0,1,0)\|_\infty\le \delta
\qquad\text{for all }t\in[\hat T_\delta,T_\delta].
\end{align*}
\end{lemma}
\begin{proof}
Choose $\alpha\in\left(1,\min\left\{k_{32},\frac1{k_{23}}\right\}\right)$, $c\coloneqq\frac{\alpha-1}{2}$ and set $\nu_\delta:=
\min\left\{\frac{\delta}{4},\frac{\alpha-1}{2\alpha}\right\}$. Define
\begin{align*}
\lambda(t):=\frac{u_3(t)}{u_2(t)^\alpha}.
\end{align*}
A direct computation gives
\begin{align*}
\frac{\lambda'(t)}{\lambda(t)}
=
(1-\alpha)
+u_1(t)(\alpha-\varepsilon k_{31})
+u_2(t)(\alpha-k_{32})
+u_3(t)(\alpha k_{23}-1).
\end{align*}
Since $\alpha<k_{32}$ and $\alpha k_{23}<1$, the last two terms are
nonpositive. If $u_1(t)\le \nu_\delta$, then
\begin{align*}
\frac{\lambda'(t)}{\lambda(t)}
\le
1-\alpha+\alpha\nu_\delta
\le
-\frac{\alpha-1}{2}
=-c.
\end{align*}
Hence $\lambda(t)\le \lambda(0)e^{-ct}$. Since $u_3(0)<1$ and $u_2(0)\ge c_0$, we have $\lambda(0)\le c_0^{-\alpha}$. Moreover, since $u_2(t)<1$, 
\begin{align*}
u_3(t)=\lambda(t)u_2(t)^\alpha
\le c_0^{-\alpha}e^{-ct}.
\end{align*}
Choose \(T_\delta^1>0\) such that $c_0^{-\alpha}e^{-cT_\delta^1}\le \frac{\delta}{4}$, then $u_3(t)\le \frac{\delta}{4}$ for all $t\ge T_\delta^1$.

Next, on \([0,T_\delta^1]\), we have
\begin{align*}
\frac{u_2'(t)}{u_2(t)}
=
1-u_1(t)-u_2(t)-k_{23}u_3(t)
\ge
-(\nu_\delta+k_{23}),
\end{align*}
because $u_2(t)<1$ and $u_3(t)<1$. Thus $u_2(T_\delta^1)
\ge c_0e^{-(\nu_\delta+k_{23})T_\delta^1}
=:a_\delta>0$.

For $t\ge T_\delta^1$, using $u_1(t)\le\nu_\delta\le\frac{\delta}{4}$ and $u_3(t)\le\frac{\delta}{4}$, {we get $\frac{u_2'(t)}{u_2(t)}\ge 1-\frac{\delta}{2}-u_2(t)\coloneqq r_{\delta}-u_{2}(t)$ with $r_{\delta}=1-\frac{\delta}{2}$}. Let \(y\) be the solution of
\begin{align*}
y^{\prime}=y(r_\delta-y),
\qquad
y(T_\delta^1)=a_\delta.
\end{align*}
By the comparison principle, $u_2(t)\ge y(t)$ for all $t\ge T_\delta^1$. The logistic equation can be solved explicitly
\[
y(t)
=
\frac{r_\delta}
{1+\left(\frac{r_\delta}{a_\delta}-1\right)
e^{-r_\delta(t-T_\delta^1)}}.
\]
We now choose
\begin{align*}
T_\delta^2
:=
2\max\left\{\ln\left[
\frac{2}{\delta}
\left(\frac{r_\delta}{a_\delta}-1\right)\right];0\right\},
\end{align*}
Then, for every $t\ge T_\delta^1+T_\delta^2$,
\begin{align*}
\left(\frac{r_\delta}{a_\delta}-1\right)
e^{-r_\delta(t-T_\delta^1)}
\le
\frac{\delta}{2}.
\end{align*}
Therefore, for all $t\ge T_{\delta}^{1}+T_{\delta}^{1}$,
\begin{align*}
u_{2}(t)\ge y(t)
\ge
\frac{r_\delta}
{1+\frac{\delta}{2}}
=
\frac{1-\frac{\delta}{2}}
{1+\frac{\delta}{2}}>1-\delta.
\end{align*}
Set $\hat T_\delta:=T_\delta^1+T_\delta^2$. Then, for every $T_\delta>\hat T_\delta$, if $0<u_1(t)\le\nu_\delta$ for all $t\in[0,T_\delta]$, we have, for all $t\in[\hat T_\delta,T_\delta]$, $u_2(t)\ge 1-\delta$ and $u_3(t)\le \frac{\delta}{4}<\delta$. Moreover \(u_2(t)<1\) and \(u_1(t)\le\nu_\delta<\delta\). Hence
\begin{align*}
\|u(t)-(0,1,0)\|_\infty\le\delta
\qquad\text{for all }t\in[\hat T_\delta,T_\delta].
\end{align*}
The proof is complete.
\end{proof}
\begin{remark}
Since $c_0e^{-(\nu_\delta+k_{23})T_\delta^{1}}<1-\frac{\delta}{2}$, the quantity $T_\delta^{2}$ is well-defined and positive.
\end{remark}

\begin{remark}\label{rem:symmetric}
The conclusion of Lemma~\ref{lem:small-u1} admits two analogous versions for the other corner states.

First, let us consider the corner state $(0,0,1)$. Choose $1<\alpha<\min\left\{k_{13},\frac{1}{k_{31}}\right\}$ and define $\lambda_3(t):=\frac{u_1(t)}{u_3(t)^\alpha}$, $\nu_\delta^{(3)}
:=
\min\left\{
\frac{\delta}{4k_{32}},
\frac{\alpha-1}{2\alpha k_{32}}
\right\}$.

If $u_2(t)\le \nu_\delta^{(3)}$ on a sufficiently long time interval, then
$\lambda_3$ decays exponentially, which forces $u_1$ to become small, while
$u_3$ satisfies a logistic lower bound and approaches $1$.

More precisely, assume that $0<c_0\le u_3(0)<1$ and $0<u_1(0)<1$. Define
\begin{align*}
T_{\delta,1}^{(3)}
:=
\frac{2}{\alpha-1}
\ln\left(\frac{4}{c_0^\alpha\delta}\right),\qquad T_{\delta,2}^{(3)}
:=
2\ln\left[
\frac{2}{\delta}
\left(
\frac{1}{c_0 e^{-(k_{31}+k_{32}\nu_\delta^{(3)})T_{\delta,1}^{(3)}}}-1
\right)
\right].
\end{align*}
Setting $\hat{T}_\delta^{(3)}:=T_{\delta,1}^{(3)}+T_{\delta,2}^{(3)}$, one obtains the following analogue of Lemma~\ref{lem:small-u1}: if
\begin{align*}
u_2(t)\le \nu_\delta^{(3)}
\qquad \text{for all}\ t\in[0,T],
\qquad
T\ge\hat{T}_\delta^{(3)},
\end{align*}
then the solution enters the $\delta$-neighborhood of $(0,0,1)$.

Next, let us consider the corner state $\left(\frac1\varepsilon,0,0\right)$. Choose $1<\beta<\frac{1}{k_{12}}$ and define $\lambda_1(t):=\frac{u_2(t)}{u_1(t)^\beta}$, $\nu_{\delta,\varepsilon}^{(1)}\coloneqq\min\left\{
\frac{\delta\varepsilon}{4},\frac{\beta-1}{2\beta k_{13}}\right\}$.

If $u_3(t)\le \nu_{\delta,\varepsilon}^{(1)}$ on a sufficiently long time interval, then $\lambda_1$ decays exponentially, so that $u_2$ becomes small, while $u_1$ satisfies a logistic lower bound and approaches $\frac1\varepsilon$.

More precisely, one may define
\begin{align*}
T_{\delta,\varepsilon,1}^{(1)}\coloneqq\frac{2}{\beta-1}
\ln\left(\frac{4}{c_0^\beta \delta \varepsilon}\right),\qquad T_{\delta,\varepsilon,2}^{(1)}\coloneqq
2\ln\left[
\frac{2}{\delta\varepsilon}\left(\frac{\tfrac1\varepsilon}{c_0 e^{-(k_{12}+k_{13}\nu_{\delta,\varepsilon}^{(1)})T_{\delta,\varepsilon,1}^{(1)}}}-1
\right)
\right].
\end{align*}
Setting $\hat{T}_{\delta,\varepsilon}^{(1)}\coloneqq T_{\delta,\varepsilon,1}^{(1)}+T_{\delta,\varepsilon,2}^{(1)}$, one obtains the following analogue of Lemma~\ref{lem:small-u1}: if
\begin{align*}
u_3(t)\le \nu_{\delta,\varepsilon}^{(1)}
\qquad \text{for all } t\in[0,T],
\qquad
T\ge \hat{T}_{\delta,\varepsilon}^{(1)},
\end{align*}
then the solution enters the $\delta$-neighborhood of
$\left(\frac1\varepsilon,0,0\right)$.
\end{remark}
\begin{proof}[Proof of Proposition~\ref{thm:corner-visits}]
Fix $\xi\in(0,1)$. Let $\delta_2^0>0$ and $0<\varepsilon_0\le 1$ be given by Theorem~\ref{thm:main}. Up to reducing $\varepsilon_0$ if necessary, we keep the same notation. Fix $0<\varepsilon\le\varepsilon_0$. We shall also reduce
$\delta_1^0(\varepsilon)$ below, if necessary, without changing the notation.

Let $u=(u_1,u_2,u_3)$ be a solution of \eqref{eq:system} satisfying
\begin{align*}
u_1(0)=1,\qquad 0<u_2(0)\le \delta_1^0(\varepsilon),
\qquad 1-\delta_2\le u_3(0)<1,
\end{align*}
with $0<\delta_2\le\delta_2^0$. We first prove that the return construction can be iterated and that the value of $u_2$ at the beginning of each pattern tends to $0$. Set $t_0:=0$ and $\eta_n:=u_2(t_n)$ whenever $t_n$ is defined. Choose
$\delta_{1,0}^+$ such that $0<u_2(0)<\delta_{1,0}^+\le \delta_1^0(\varepsilon)$. Assume that $t_n$ and $\delta_{1,n}^+$ have been constructed and satisfy
\begin{align*}
u_1(t_n)=1,\qquad
0<\eta_n:=u_2(t_n)<\delta_{1,n}^+\le \delta_1^0(\varepsilon),
\qquad
u_3(t_n)\ge 1-\delta_2^0.
\end{align*}
Then Theorem~\ref{thm:main}, applied to the shifted solution
$s\mapsto u(t_n+s)$ with $\delta_1^+$ replaced by $\delta_{1,n}^+$, gives a
finite return time
\begin{align*}
t_{n+1}
:=
\inf\{t>t_n:\ u_1(t)=1,\ u_1'(t)<0\}.
\end{align*}
Moreover, the upper estimate on $u_2(t_{n+1})$ in
Theorem~\ref{thm:main} gives constants $C_\varepsilon>0$ and
$p_\varepsilon>1$, depending on $\varepsilon$ and on the coefficients, such
that
\begin{align*}
u_2(t_{n+1})
\le C_\varepsilon(\delta_{1,n}^+)^{p_\varepsilon}\coloneqq \delta_{1,n+1}^{+}.
\end{align*}
Here $p_\varepsilon>1$ holds after reducing $\varepsilon_0$ if necessary. Since $p_\varepsilon>1$, we may reduce $\delta_1^0(\varepsilon)$ so that
\begin{align*}
C_\varepsilon \rho^{p_\varepsilon-1}\le \frac12
\qquad
\text{for all }0<\rho\le \delta_1^0(\varepsilon).
\end{align*}
Therefore, whenever $0<\delta_{1,n}^+\le\delta_1^0(\varepsilon)$, we have $\delta_{1,n+1}^+=\delta_{1,n}^+
\left(C_\varepsilon(\delta_{1,n}^+)^{p_\varepsilon-1}
\right)\le\frac12\,\delta_{1,n}^{+}$. In particular,
\begin{align*}
0<\delta_{1,n+1}^+\le \delta_{1,n}^+\le \delta_1^0(\varepsilon),
\end{align*}
and the induction can continue.

It remains to check that the condition on the third component is preserved. By the estimate on $u_3$ in Theorem~\ref{thm:main}, there exists $a>0$,
depending only on the coefficients, such that
\begin{align*}
1-u_3(t_{n+1})\le \varepsilon^a.
\end{align*}
After reducing $\varepsilon_0$ if necessary, we may assume that $\varepsilon^a\le \delta_2^0$ for every $0<\varepsilon\le\varepsilon_0$. Thus $u_3(t_{n+1})\ge 1-\delta_2^0$. This closes the induction. Hence all return times $t_n$ are well-defined. Moreover, $\delta_{1,n}^+\le 2^{-n}\delta_{1,0}^+$, and therefore $\delta_{1,n}^+\to0$.

For each $n$, denote by
\begin{align*}
T_{11,n}<T_{12,n}<T_{1,n}<T_{21,n}<T_{2,n}
<T_{31,n}<T_{32,n}<t_{n+1},
\end{align*}
the corresponding absolute transition times obtained by applying the lemmas in the proof of Theorem~\ref{thm:main} to the shifted solution
$s\mapsto u(t_n+s)$. We now prove that, for all sufficiently large $n$, the solution visits successively prescribed neighborhoods of $E_3$, $E_2$, and $E_1$ during the interval $(t_n,t_{n+1})$. The only fact from the induction used below is $\delta_{1,n}^+\to0$.

First, we prove the visit near $E_3=(0,0,1)$. Let
$\nu_\xi^{(3)}$ and $\widehat T_\xi^{(3)}$ be given by the first analogue in Remark~\ref{rem:symmetric}, with $\delta=\xi$. Replacing, if necessary, the threshold $\nu_2$ in Lemma~\ref{lem:main-12} by $\widetilde\nu_2:=\min\{\nu_2,\nu_\xi^{(3)}\}$, we obtain a time $\widetilde T_{12,n}\in[T_{11,n},T_{12,n}]$ such that
\begin{align*}
u_2(t)\le \nu_\xi^{(3)}
\qquad
\text{for all }t\in[T_{11,n},\widetilde T_{12,n}],
\end{align*}
while $u_3$ stays bounded below by a positive constant on this interval.
Moreover, the proof of Lemma~\ref{lem:main-12} gives a constant $C>0$,
independent of $n$, such that
\begin{align*}
\widetilde T_{12,n}-T_{11,n}
\ge
\ln\left(\frac{\widetilde\nu_2}{C\delta_{1,n}^+}\right).
\end{align*}
Since $\delta_{1,n}^+\to0$, it follows that $\widetilde T_{12,n}-T_{11,n}\to+\infty$. Thus, for all sufficiently large $n$, this length is at least
$\widehat T_\xi^{(3)}$. Applying the first analogue in
Remark~\ref{rem:symmetric}, we obtain a time
$\tau_{3,n}^\xi\in(T_{11,n},\widetilde T_{12,n})$ such that
\begin{align*}
\|u(\tau_{3,n}^\xi)-E_3\|_\infty\le \xi.
\end{align*}

Second, we prove the visit near $E_2=(0,1,0)$. Let $\nu_\xi$ and
$\widehat T_\xi$ be given by Lemma~\ref{lem:small-u1}, with $\delta=\xi$.
Replacing, if necessary, $\nu_4$ in Lemma~\ref{lem:main-21} by $\widetilde\nu_4:=\min\{\nu_4,\nu_\xi\}$, we obtain a time $\widetilde T_{21,n}\in[T_{1,n},T_{21,n}]$ such that
\begin{align*}
u_1(t)\le \nu_\xi
\qquad
\text{for all }t\in[T_{1,n},\widetilde T_{21,n}],
\end{align*}
while $u_2$ stays bounded below by a positive constant on this interval.
Furthermore, the proof of Lemma~\ref{lem:main-21} gives constants $C>0$ and $\theta>0$, independent of $n$, such that
\begin{align*}
\widetilde T_{21,n}-T_{1,n}
\ge
\ln\left(\dfrac{\widetilde\nu_4}{C\left(\delta_{1,n}^{+}\right)^{\theta}}\right).
\end{align*}
Since $\delta_{1,n}^+\to0$, we have $\widetilde T_{21,n}-T_{1,n}\to+\infty$. Therefore, for all sufficiently large $n$, this length is at least $\widehat T_\xi$. Applying Lemma~\ref{lem:small-u1}, we obtain a time
$\tau_{2,n}^\xi\in(T_{1,n},\widetilde T_{21,n})$ such that
\begin{align*}
\|u(\tau_{2,n}^\xi)-E_2\|_\infty\le \xi.
\end{align*}

Third, we prove the visit near $E_1=(1/\varepsilon,0,0)$. Let
$\nu_{\xi,\varepsilon}^{(1)}$ and $\widehat T_{\xi,\varepsilon}^{(1)}$ be
given by the second analogue in Remark~\ref{rem:symmetric}, with
$\delta=\xi$. Replacing, if necessary, $\nu_6$ in Lemma~\ref{lem:main-32} by $\widetilde\nu_6:=\min\{\nu_6,\nu_{\xi,\varepsilon}^{(1)}\}$, we obtain a time $\widetilde T_{32,n}\in[T_{31,n},T_{32,n}]$ such that
\begin{align*}
u_3(t)\le \nu_{\xi,\varepsilon}^{(1)}
\qquad
\text{for all }t\in[T_{31,n},\widetilde T_{32,n}],
\end{align*}
while $u_1$ stays bounded below by a positive multiple of $1/\varepsilon$ on
this interval. Moreover, the proof of Lemma~\ref{lem:main-32} gives constants $\alpha>0$ and $\beta>0$, independent of $n$, such that
\begin{align*}
\widetilde T_{32,n}-T_{31,n}
\ge\ln\left(\dfrac{\tilde\nu_{6}\,\varepsilon^{\alpha}}{\left(\delta_{1,n}^{+}\right)^{\beta}}\right).
\end{align*}
Since $\delta_{1,n}^+\to0$, we get $\widetilde T_{32,n}-T_{31,n}\to+\infty$. Thus, for all sufficiently large $n$, this length is at least $\widehat T_{\xi,\varepsilon}^{(1)}$. Applying the second analogue in Remark~\ref{rem:symmetric}, we obtain a time
$\tau_{1,n}^\xi\in(T_{31,n},\widetilde T_{32,n})$ such that
\begin{align*}
\|u(\tau_{1,n}^\xi)-E_1\|_\infty\le \xi.
\end{align*}

Finally, by construction,
\begin{align*}
T_{11,n}<\widetilde T_{12,n}\le T_{12,n}
<T_{1,n}<\widetilde T_{21,n}\le T_{21,n}<T_{2,n}<T_{31,n}
<\widetilde T_{32,n}\le T_{32,n}.
\end{align*}
Hence, for all sufficiently large $n$, $\tau_{3,n}^\xi<\tau_{2,n}^\xi<\tau_{1,n}^\xi$. Choosing $N_\xi$ large enough so that the three conclusions above hold for
every $n\ge N_\xi$, we obtain, for every $n\ge N_\xi$,
\begin{align*}
\left\|u(\tau_{3,n}^\xi)-E_{3}\right\|_\infty\le \xi,
\qquad
\left\|u(\tau_{2,n}^\xi)-E_{2}\right\|_\infty\le \xi,
\qquad
\left\|u(\tau_{1,n}^\xi)-E_{1}\right\|_\infty\le \xi.
\end{align*}
This proves Proposition~\ref{thm:corner-visits}.
\end{proof}

\subsection{Proof of Theorem \ref{thm:heteroclinic-connections}}
We begin by stating several lemmas that will be useful for the proof of Theorem \ref{thm:heteroclinic-connections}.

\begin{lemma}\label{lem:small-transverse-components}
For every \(\xi\in(0,1)\), there exist \(N_\xi\ge1\), constants
\(C>0\), \(\alpha>0\), and \(\gamma>0\), independent of \(n\), such that, for
every \(n\ge N_\xi\), the following estimates hold:
\begin{align*}
\sup_{t\in[\tau_{3,n}^{\xi},\,\tau_{2,n}^{\xi}]}u_1(t)
\le
\max\left\{\xi,\ C(\delta_{1,n}^{+})^\alpha\right\},
\end{align*}
\begin{align*}
\sup_{t\in[\tau_{2,n}^{\xi},\,\tau_{1,n}^{\xi}]}u_3(t)
\le
\max\left\{\xi,\ C\varepsilon^{-\gamma}(\delta_{1,n}^{+})^\alpha\right\},
\end{align*}
and
\begin{align*}
\sup_{t\in[\tau_{1,n}^{\xi},\,\tau_{3,n+1}^{\xi}]}u_2(t)
\le\max\left\{\xi,\ C\left(\frac{\delta_{1,n}^{+}}{\varepsilon}\right)^{\alpha(1+\frac1\varepsilon)}
\right\}.
\end{align*}
Consequently, there exist sequences $\xi_k\downarrow0$ and
$n_k\to+\infty$ such that
\begin{align*}
\sup_{t\in[\tau_{3,n_k}^{\xi_k},\,\tau_{2,n_k}^{\xi_k}]}u_1(t)\to0,
\end{align*}
\begin{align*}
\sup_{t\in[\tau_{2,n_k}^{\xi_k},\,\tau_{1,n_k}^{\xi_k}]}u_3(t)\to0,
\end{align*}
and
\begin{align*}
\sup_{t\in[\tau_{1,n_k}^{\xi_k},\,\tau_{3,n_k+1}^{\xi_k}]}u_2(t)\to0.
\end{align*}
\end{lemma}
\begin{proof}
We prove the first estimate; the other two follow from the same argument,
using the corresponding estimates from Lemmas~\ref{lem:main-21},
\ref{lem:main-31}, and \ref{lem:main-32}.

By the construction in the proof of Proposition~\ref{thm:corner-visits}, for all $n$ sufficiently large,
\begin{align*}
\tau_{3,n}^{\xi}<\widetilde T_{12,n}
\le T_{12,n}< T_{1,n}<\tau_{2,n}^{\xi}<\widetilde T_{21,n}.
\end{align*}
Hence
\begin{align*}
[\tau_{3,n}^{\xi},\tau_{2,n}^{\xi}]
\subset
[\tau_{3,n}^{\xi},\widetilde T_{12,n}]
\cup
[\widetilde T_{12,n},T_{12,n}]
\cup
[T_{12,n},T_{1,n}]
\cup
[T_{1,n},\tau_{2,n}^{\xi}].
\end{align*}
On the first interval, the orbit is inside the $\xi$-neighborhood of
$E_3=(0,0,1)$, and therefore $u_1(t)\le \xi$. On $[T_{11,n},T_{12,n}]$, Lemma~\ref{lem:main-12} gives $u_3(t)>\frac12\left(1+\frac1{k_{13}}\right)$, so
\begin{align*}
\frac{u_1'(t)}{u_1(t)}=1-\varepsilon u_1(t)-k_{12}u_2(t)-k_{13}u_3(t)
<-\frac{k_{13}-1}{2}<0.
\end{align*}
Thus $u_1$ is decreasing on this interval, and consequently
\begin{align*}
\sup_{t\in[\widetilde T_{12,n},T_{12,n}]}u_1(t)
\le u_1(\widetilde T_{12,n})\le \xi.
\end{align*}
Next, Lemma~\ref{lem:main-12} gives constants $\alpha>0$ and $C>0$ such
that
\begin{align*}
u_1(T_{12,n})\le C(\delta_{1,n}^+)^\alpha.
\end{align*}
Since $u_1'/u_1\le1$ and $T_{1,n}-T_{12,n}$ is uniformly bounded by
Proposition~\ref{prop:main-1}, we also have
\begin{align*}
\sup_{t\in[T_{12,n},T_{1,n}]}u_1(t)
\le
C^{\prime}(\delta_{1,n}^+)^\alpha.
\end{align*}
Finally, by the construction of $\widetilde T_{21,n}$, $u_1(t)\le \nu_\xi\le \frac{\xi}{4}$ for all $t\in[T_{1,n},\widetilde T_{21,n}]$.
Combining these bounds yields
\begin{align*}
\sup_{t\in[\tau_{3,n}^{\xi},\,\tau_{2,n}^{\xi}]}u_1(t)
\le
\max\left\{\xi,\ C(\delta_{1,n}^{+})^\alpha\right\}.
\end{align*}

The estimates for \(u_3\) on
\([\tau_{2,n}^{\xi},\tau_{1,n}^{\xi}]\) and for \(u_2\) on
\([\tau_{1,n}^{\xi},\tau_{3,n+1}^{\xi}]\) are obtained in the same way. One
uses respectively the decompositions
\[
[\tau_{2,n}^{\xi},\tau_{1,n}^{\xi}]
\subset
[\tau_{2,n}^{\xi},\widetilde T_{21,n}]
\cup
[\widetilde T_{21,n},T_{21,n}]
\cup
[T_{21,n},T_{2,n}]
\cup
[T_{2,n},T_{31,n}]
\cup
[T_{31,n},\tau_{1,n}^{\xi}],
\]
and
\[
[\tau_{1,n}^{\xi},\tau_{3,n+1}^{\xi}]
\subset
[\tau_{1,n}^{\xi},\widetilde T_{32,n}]
\cup
[\widetilde T_{32,n},T_{3,n}]
\cup
[T_{3,n},T_{11,n+1}]
\cup
[T_{11,n+1},\tau_{3,n+1}^{\xi}].
\]
On the parts close to the corner, the corresponding transverse component is
bounded by \(\xi\). On the monotonicity intervals, the same component is
decreasing. On the remaining intervals, the estimates of
Lemmas~\ref{lem:main-21}, \ref{lem:main-31}, \ref{lem:main-32}, together with
Gronwall's inequality and the uniform bounds on the relevant time lengths,
give
\[
\sup_{t\in[\tau_{2,n}^{\xi},\,\tau_{1,n}^{\xi}]}u_3(t)
\le
\max\left\{\xi,\ C\varepsilon^{-\gamma}(\delta_{1,n}^{+})^\alpha\right\},
\]
and
\[
\sup_{t\in[\tau_{1,n}^{\xi},\,\tau_{3,n+1}^{\xi}]}u_2(t)
\le
\max\left\{\xi,\ 
C\left(\frac{\delta_{1,n}^{+}}{\varepsilon}\right)^{\alpha(1+\frac1\varepsilon)}
\right\}.
\]

It remains to choose a subsequence. Let \(\xi_k\downarrow0\). Since
\(\delta_{1,n}^+\to0\), we may choose \(n_k\to+\infty\) such that all three
remainders in the above estimates converge to \(0\). The three convergence
statements follow immediately.
\end{proof}
\begin{proof}[Proof of Theorem~\ref{thm:heteroclinic-connections}]
We prove only the existence of the connection from
\(E_3=(0,0,1)\) to \(E_2=(0,1,0)\). The other two connections are obtained by
the same argument.

Let $\xi_k\downarrow0$. By Proposition~\ref{thm:corner-visits} and
Lemma~\ref{lem:small-transverse-components}, we may choose $n_k\to+\infty$ such that
\begin{align*}
\|u(\tau_{3,n_k}^{\xi_k})-E_3\|_\infty\le \xi_k,
\qquad
\|u(\tau_{2,n_k}^{\xi_k})-E_2\|_\infty\le \xi_k,
\end{align*}
and
\begin{align*}
\sup_{t\in[\tau_{3,n_k}^{\xi_k},\,\tau_{2,n_k}^{\xi_k}]}u_1(t)
\longrightarrow0.
\end{align*}
Let $T_{1,n_k}$ be the intermediate time given by
Proposition~\ref{prop:main-1}. There exists $\eta_1\in(0,1)$, independent of $k$, such that, for all $k$ large enough,
\begin{align*}
\eta_1\le u_2(T_{1,n_k})\le 1-\eta_1,
\qquad
\eta_1\le u_3(T_{1,n_k})\le 1-\eta_1.
\end{align*}

We first show that \(T_{1,n_k}\) lies far from both endpoints of the transition.
Since \(u\) is bounded, there exists \(M>0\) such that
\begin{align*}
\left|\frac{u_2'}{u_2}\right|+
\left|\frac{u_3'}{u_3}\right|
\le M
\qquad\text{on }\mathbb R_+.
\end{align*}
Using \(u_2(\tau_{3,n_k}^{\xi_k})\le\xi_k\) and
\(u_2(T_{1,n_k})\ge\eta_1\), we get
\begin{align*}
T_{1,n_k}-\tau_{3,n_k}^{\xi_k}
\ge
\frac1M\log\left(\frac{\eta_1}{\xi_k}\right)
\longrightarrow+\infty.
\end{align*}
Similarly, using $u_3(\tau_{2,n_k}^{\xi_k})\le\xi_k$ and $u_3(T_{1,n_k})\ge\eta_1$, we obtain
\begin{align*}
\tau_{2,n_k}^{\xi_k}-T_{1,n_k}
\ge
\frac1M\log\left(\frac{\eta_1}{\xi_k}\right)
\longrightarrow+\infty.
\end{align*}
Define the shifted solutions $v_k(t)\coloneqq u(t+T_{1,n_k})$. The previous
estimates imply that $v_k$ is defined on every compact interval of $\mathbb R$
for $k$ large enough. Since $u$ is bounded and the vector field in
\eqref{eq:system} is smooth, $(v_k)_k$ is uniformly bounded and equicontinuous
on compact intervals. By Arzelà--Ascoli and a diagonal extraction, there
exists a subsequence, still denoted by $v_k$, and a global solution
$v^\infty:\mathbb R\to\mathbb R_+^3$ of \eqref{eq:system} such that $v_k\to v^\infty$ in $C^1_{\mathrm{loc}}(\mathbb R)$.

We claim that the first component of $v^\infty$ vanishes, i.e. $v^\infty_1\equiv0$. Indeed, for every fixed $A>0$, the interval
$[T_{1,n_k}-A,T_{1,n_k}+A]$ is contained in
$[\tau_{3,n_k}^{\xi_k},\tau_{2,n_k}^{\xi_k}]$ for $k$ large enough. Hence
\begin{align*}
\sup_{|t|\le A}v_{k,1}(t)
\le
\sup_{s\in[\tau_{3,n_k}^{\xi_k},\,\tau_{2,n_k}^{\xi_k}]}u_1(s)
\longrightarrow0.
\end{align*}
Passing to the limit gives $v^\infty_1(t)=0$ on $[-A,A]$. Since $A$ is
arbitrary, $v^\infty_1\equiv0$ on $\mathbb R$. Thus $v^\infty$ is a global
solution of the reduced system on the invariant face $\{u_1=0\}$:
\begin{align*}
\begin{cases}
(v^\infty_2)^{\prime}=v^\infty_2(1-v^\infty_2-k_{23}v^\infty_3),\\
(v^\infty_3)^{\prime}=v^\infty_3(1-k_{32}v^\infty_2-v^\infty_3).
\end{cases}
\end{align*}
Moreover,
\begin{align*}
\eta_1\le v^\infty_2(0)\le1-\eta_1,
\qquad
\eta_1\le v^\infty_3(0)\le1-\eta_1,
\end{align*}
so $v^\infty(0)$ lies in the interior of the positive quadrant of the face
$\{u_1=0\}$.

For the reduced planar system, $0<k_{23}<1<k_{32}$. Hence there is no positive
equilibrium. By the Bendixson--Dulac criterion, there is no periodic orbit in
$(0,\infty)^2$. Therefore the alpha- and omega-limit sets of the interior orbit
$v^\infty$ lie on the boundary. Since $E_3=(0,0,1)$ is the boundary equilibrium
that repels into the interior and $E_2=(0,1,0)$ is the boundary equilibrium
that attracts the interior orbit on this face, we obtain
\begin{align*}
v^\infty(t)\to E_3
\qquad\text{as }t\to-\infty,
\qquad
v^\infty(t)\to E_2
\qquad\text{as }t\to+\infty.
\end{align*}
Thus $v^\infty$ is a heteroclinic connection from $E_3$ to $E_2$.

The two remaining connections are obtained by applying the same compactness argument to the transitions $E_2\to E_1$ and $E_1\to E_3$,
using the intermediate times $T_{2,n}$ and $T_{33,n}$, together with the smallness estimates in Lemma~\ref{lem:small-transverse-components}. This gives global solutions $v^{21}$ and $v^{13}$ satisfying
\begin{align*}
v^{21}_3\equiv0,\qquad v^{13}_2\equiv0,
\end{align*}
and
\begin{align*}
v^{21}(t)\to E_2\quad\text{as}\ t\to-\infty,
\qquad
v^{21}(t)\to E_1\quad\text{as}\ t\to+\infty,
\end{align*}
\begin{align*}
v^{13}(t)\to E_1\quad\text{as}\ t\to-\infty,
\qquad
v^{13}(t)\to E_3\quad\text{as}\ t\to+\infty.
\end{align*}
Therefore the system admits the heteroclinic cycle
\begin{align*}
E_3\longrightarrow E_2\longrightarrow E_1\longrightarrow E_3.
\end{align*}
\end{proof}
\subsection{Proof of Theorem \ref{thm:assymp-sol-hetero}}
\begin{proof}[Proof of Theorem \ref{thm:assymp-sol-hetero}]
We only prove \eqref{eq:geom-conv-32}; the other two convergences are obtained in the same way.

Assume by contradiction that \eqref{eq:geom-conv-32} is false. Then there exist $A>0$, $\eta>0$, a subsequence $(n_j)_j$, and points
$s_j\in[-A,A]$ such that $d\bigl(u(T_{1,n_j}+s_j),\Gamma_{32}\bigr)\ge \eta$ for all $j$. Define the shifted functions
\begin{align*}
w_j(s):=u(T_{1,n_j}+s).
\end{align*}
Since $u$ is bounded and the vector field in \eqref{eq:system} is smooth on
bounded sets, the sequence $(w_j)_j$ is uniformly bounded and equicontinuous
on every compact interval of $\mathbb R$. Hence, by the Arzelà--Ascoli theorem and a diagonal argument, up to extracting a subsequence, {$w_j\to w^\infty$ in $C^1_{\mathrm{loc}}(\mathbb R)$, where $w^\infty$ is a global solution of \eqref{eq:system}. We write
\begin{align*}
w^\infty(s)=\bigl(w^\infty_1(s),w^\infty_2(s),w^\infty_3(s)\bigr).
\end{align*}
We first show that $w^\infty_1\equiv0$}. Fix $B>0$ and let $\rho>0$. By
Lemma~\ref{lem:small-transverse-components}, for all $n$ large enough,
\begin{align*}
\sup_{t\in[\tau_{3,n}^{\rho},\tau_{2,n}^{\rho}]}u_1(t)
\le
\max\left\{\rho,C(\delta_{1,n}^+)^\alpha\right\}.
\end{align*}
Moreover, the time-separation estimates obtained in the proof of
Theorem~\ref{thm:heteroclinic-connections} give
\begin{align*}
T_{1,n}-\tau_{3,n}^{\rho}\to+\infty,
\qquad
\tau_{2,n}^{\rho}-T_{1,n}\to+\infty.
\end{align*}
Thus, for $j$ large enough, $[T_{1,n_j}-B,T_{1,n_j}+B]\subset
[\tau_{3,n_j}^{\rho},\tau_{2,n_j}^{\rho}]$. It follows that
\begin{align*}
\sup_{|s|\le B}w_{j,1}(s)
\le
\max\left\{\rho,C(\delta_{1,n_j}^+)^\alpha\right\}.
\end{align*}
Letting $j\to+\infty$ and then $\rho\to0$, we get $\sup_{|s|\le B}w_{j,1}(s)\to0$. Passing to the limit gives $w_1^{\infty}(s)=0$ for all $s\in[-B,B]$. Since $B>0$ is
arbitrary, $w_1^{\infty}\equiv0$ on $\mathbb R$.

Therefore $w^\infty$ is a solution of the reduced system on the invariant
face $\{u_1=0\}$:
\begin{align*}
\begin{cases}
(w^\infty_2)^{\prime}=w_2(1-w^\infty_2-k_{23}w^\infty_3),\\
(w^\infty_3)^{\prime}=w_3(1-k_{32}w^\infty_2-w^\infty_3).
\end{cases}
\end{align*}
Moreover, by the choice of the intermediate time $T_{1,n}$, there exists
$\eta_1\in(0,1)$ such that
\begin{align*}
\eta_1\le w^\infty_2(0)\le1-\eta_1,
\qquad
\eta_1\le w^\infty_3(0)\le1-\eta_1.
\end{align*}
Hence $w(0)$ lies in the interior of the positive quadrant.

Since $0<k_{23}<1<k_{32}$, the reduced planar system has no positive
equilibrium. By the Bendixson--Dulac criterion, it has no periodic orbit in
$(0,\infty)^2$. Therefore the complete interior orbit $w$ is the unique
geometric heteroclinic connection from $E_3$ to $E_2$ on this invariant face. Consequently, $w^\infty(\mathbb R)=\Gamma_{32}$.

Since $s_j\in[-A,A]$ and $w_j\to w^\infty$ uniformly on $[-A,A]$, we obtain
\begin{align*}
d\bigl(w_j(s_j),\Gamma_{32}\bigr)\le
\|w_j(s_j)-w^\infty(s_j)\|_\infty\le
\sup_{|s|\le A}\|w_j(s)-w^\infty(s)\|_\infty
\xrightarrow[j\to+\infty]{}0.
\end{align*}
This is a contradiction. Hence \eqref{eq:geom-conv-32}
holds.
\end{proof}
\subsection{Proof of Theorem \ref{thm:omega-limit-cycle}}
\begin{lemma}\label{lem:close-gamma32}
Set $\Gamma_{32}^{\star}:=\Gamma_{32}\cup\{E_3,E_2\}$. Then, for every $\xi>0$, there exist $\eta>0$ and $N\ge1$ such that, for
every $n\ge N$,
\begin{align*}
d(u(t),\Gamma_{32}^{\star})\le \xi,
\qquad
\forall t\in[\tau_{3,n}^{\eta},\tau_{2,n}^{\eta}].
\end{align*}
In particular, since $\Gamma_{32}^{\star}\subset\Gamma$, one also has
\begin{align*}
d(u(t),\Gamma)\le \xi,
\qquad
\forall t\in[\tau_{3,n}^{\eta},\tau_{2,n}^{\eta}].
\end{align*}
\end{lemma}

\begin{proof}

Fix $0<\xi<1$. Choose $1<\alpha<\min\left\{k_{32},\frac1{k_{23}}\right\}$. Then choose $\eta>0$ small enough so that
\begin{align*}
\eta<\xi,\qquad
\eta<\frac{1-k_{23}}{2},\qquad
\eta<(1-k_{23})\xi,\qquad
(k_{31}+k_{32})\eta<\xi,
\end{align*}
and
\begin{align*}
\eta<\frac{\alpha-1}{2\alpha},
\qquad
\frac{\eta}{(1-\eta)^\alpha}<\xi.
\end{align*}
Let $\gamma^{32}$ be a parametrization of the heteroclinic orbit $\Gamma_{32}$ from
$E_3$ to $E_2$. Since
\begin{align*}
\gamma^{32}(s)\to E_3
\qquad\text{as }s\to-\infty,\qquad \gamma^{32}(s)\to E_2
\qquad\text{as }s\to+\infty,
\end{align*}
there exists $T>0$ such that
\begin{align*}
\|\gamma^{32}(-T)-E_3\|_\infty<\frac{\eta}{2},
\qquad
\|\gamma^{32}(T)-E_2\|_\infty<\frac{\eta}{2}.
\end{align*}
By the local convergence to the branch $\Gamma_{32}$, after choosing the
parametrization $\gamma^{32}$ consistently with the shifts $T_{1,n}$, there exists
$N\ge1$ such that, for every $n\ge N$,
\begin{align*}
\sup_{|s|\le T}
\|u(T_{1,n}+s)-\gamma^{32}(s)\|_\infty<\frac{\eta}{2}.
\end{align*}
Increasing $N$ if necessary, the time-separation estimates give
\begin{align*}
\tau_{3,n}^{\eta}<T_{1,n}-T<T_{1,n}+T<\tau_{2,n}^{\eta}
\end{align*}
for every $n\ge N$. Hence
\begin{align*}
[\tau_{3,n}^{\eta},\tau_{2,n}^{\eta}]
=
[\tau_{3,n}^{\eta},T_{1,n}-T]
\cup
[T_{1,n}-T,T_{1,n}+T]
\cup
[T_{1,n}+T,\tau_{2,n}^{\eta}].
\end{align*}
On the middle interval, if $t=T_{1,n}+s$, with $|s|\le T$, then
\begin{align*}
d(u(t),\Gamma_{32})
\le
\|u(T_{1,n}+s)-\gamma^{32}(s)\|_\infty
<\frac{\eta}{2}<\xi.
\end{align*}
Thus $d(u(t),\Gamma_{32}^{\star})\le\xi$ for all $t\in[T_{1,n}-T,T_{1,n}+T]$.

We next consider the left interval $[\tau_{3,n}^{\eta},T_{1,n}-T]$. By Lemma~\ref{lem:small-transverse-components}, after increasing $N$ if
necessary,
\begin{align*}
u_1(t)\le\eta
\qquad
\forall t\in[\tau_{3,n}^{\eta},\tau_{2,n}^{\eta}].
\end{align*}
We claim that $u_2(t)\le\eta$ on
$[\tau_{3,n}^{\eta},T_{1,n}-T]$. Indeed, $u_2(\tau_{3,n}^{\eta})\le\eta$ and $u_2(T_{1,n}-T)<\eta$, because both endpoints are in the $\eta$-neighborhood of $E_3$. If $u_2$ were
larger than $\eta$ somewhere in between, then at the first downward crossing
$t_1$ of the level $\eta$, one would have $u_2(t_1)=\eta$ and
$u_2'(t_1)\le0$. But, using $u_1(t_1)\le\eta$ and $u_3(t_1)\le1$,
\begin{align*}
\frac{u_2'(t_1)}{u_2(t_1)}
=
1-u_1(t_1)-u_2(t_1)-k_{23}u_3(t_1)
\ge
1-2\eta-k_{23}>0,
\end{align*}
which is a contradiction. Hence $u_2(t)\le\eta$ on this interval. Therefore, on the left interval,
\begin{align*}
\frac{u_3'(t)}{u_3(t)}
\ge
1-(k_{31}+k_{32})\eta-u_3(t)\ge
1-\xi-u_3(t).
\end{align*}
Since $u_3(\tau_{3,n}^{\eta})\ge1-\eta\ge1-\xi$, comparison with
$z'=z(1-\xi-z)$ gives $u_3(t)\ge1-\xi$ for all $t\in[\tau_{3,n}^{\eta},T_{1,n}-T]$. Together with $u_1(t)\le\eta<\xi$ and $u_2(t)\le\eta<\xi$, this yields
\begin{align*}
\|u(t)-E_3\|_\infty\le\xi
\qquad
\forall t\in[\tau_{3,n}^{\eta},T_{1,n}-T].
\end{align*}
Since $E_3\in\Gamma_{32}^{\star}$, we obtain $d(u(t),\Gamma_{32}^{\star})\le\xi$ on $[\tau_{3,n}^{\eta},T_{1,n}-T]$.

It remains to consider the right interval $[T_{1,n}+T,\tau_{2,n}^{\eta}]$. At its left endpoint, $u(T_{1,n}+T)$ belongs to the $\eta$-neighborhood of
$E_2$. Hence $u_2(T_{1,n}+T)\ge1-\eta$ and $u_3(T_{1,n}+T)\le\eta$.
Moreover $u_1(t)\le\eta$ on the whole interval. As in the proof of
Lemma~\ref{lem:small-u1}, define
\begin{align*}
\lambda(t):=\frac{u_3(t)}{u_2(t)^\alpha}.
\end{align*}
Since $1<\alpha<\min\{k_{32},1/k_{23}\}$ and
$u_1(t)\le\eta<(\alpha-1)/(2\alpha)$, the same computation gives
\begin{align*}
\frac{\lambda'(t)}{\lambda(t)}
\le
-\frac{\alpha-1}{2}.
\end{align*}
Consequently, $\lambda(t)\le \lambda(T_{1,n}+T)
\le
\frac{\eta}{(1-\eta)^\alpha}
<\xi$. Since $u_2(t)\le1$, we get $u_3(t)\le \xi$ on $[T_{1,n}+T,\tau_{2,n}^{\eta}]$.

Now
\begin{align*}
\frac{u_2'(t)}{u_2(t)}\ge
1-\eta-k_{23}\xi-u_2(t)
>
1-\xi-u_2(t),
\end{align*}
by the choice of $\eta$. Since $u_2(T_{1,n}+T)\ge1-\eta>1-\xi$,
comparison with $y'=y(1-\xi-y)$ gives
\begin{align*}
u_2(t)\ge1-\xi
\qquad
\forall t\in[T_{1,n}+T,\tau_{2,n}^{\eta}].
\end{align*}
Combining $u_1(t)\le\eta<\xi$, $u_3(t)\le\xi$, and $u_2(t)\ge1-\xi$, we obtain
\begin{align*}
\|u(t)-E_2\|_\infty\le\xi
\qquad
\forall t\in[T_{1,n}+T,\tau_{2,n}^{\eta}].
\end{align*}
Since $E_2\in\Gamma_{32}^{\star}$, $d(u(t),\Gamma_{32}^{\star})\le\xi$ for all $t\in[T_{1,n}+T,\tau_{2,n}^{\eta}]$. Combining the estimates on the three subintervals gives
\begin{align*}
d(u(t),\Gamma_{32}^{\star})\le\xi,
\qquad
\forall t\in[\tau_{3,n}^{\eta},\tau_{2,n}^{\eta}].
\end{align*}
Since $\Gamma_{32}^{\star}\subset\Gamma$, the last assertion follows.
\end{proof}
\begin{proof}[Proof of Theorem~\ref{thm:omega-limit-cycle}]
We first prove that $d(u(t),\Gamma)\to0$. By
Lemma~\ref{lem:close-gamma32} and its two analogous versions for
$\Gamma_{21}$ and $\Gamma_{13}$, for every $\xi>0$, there exist $\eta>0$ and
$N\ge1$ such that, for every $n\ge N$,
\begin{align*}
d(u(t),\Gamma)\le\xi
\qquad
\forall t\in[\tau_{3,n}^{\eta},\tau_{3,n+1}^{\eta}].
\end{align*}
Since the intervals $[\tau_{3,n}^{\eta},\tau_{3,n+1}^{\eta}]$ cover all
sufficiently large times, we obtain $d(u(t),\Gamma)\le\xi$ for all sufficiently large $t$. As $\xi>0$ is arbitrary, this proves
\begin{align*}
d(u(t),\Gamma)\xrightarrow[t\to+\infty]{}0.
\end{align*}

We now identify the omega-limit set. Let $y\in\omega(u)$. Then there exists a
sequence $t_k\to+\infty$ such that $u(t_k)\to y$. Since
$d(u(t),\Gamma)\to0$, we get $d(y,\Gamma)=0$. The set $\Gamma$ is closed, and therefore $y\in\Gamma$. Hence $\omega(u)\subset\Gamma$.

Conversely, we prove that $\Gamma\subset\omega(u)$. The three equilibria
$E_1,E_2,E_3$ belong to $\omega(u)$ by Proposition~\ref{thm:corner-visits}.
Let $y\in\Gamma_{32}$. Then there exists $s_0\in\mathbb R$ such that $y=v^{32}(s_0)$. By the local convergence, $u(T_{1,n}+s_0)\to v^{32}(s_0)=y$. Since $T_{1,n}+s_0\to+\infty$, this implies $y\in\omega(u)$. Thus $\Gamma_{32}\subset\omega(u)$. The same argument, using the convergences around $T_{2,n}$ and $T_{33,n}$,
gives $\Gamma_{21}\subset\omega(u)$ and $\Gamma_{13}\subset\omega(u)$.

Therefore $\Gamma\subset\omega(u)$. Combining the two inclusions, we conclude that $\omega(u)=\Gamma$.
\end{proof}
\appendix
\section{Parameters of the numerical computations}
\label{sec:parameters}

The parameters used for the numerical computations in Figures \ref{fig:two-patterns}--\ref{Figure3} are as follows. We use 
$k_{12}=0.1$, $k_{13}=4$, $k_{23}=0.1$, $k_{31}=0.1$, $k_{32}=5$ in all the simulations. 
We list the remaining parameters for each figure below. 

\noindent   Figure \ref{fig:two-patterns} (solution of \eqref{eq:main-1}): $\varepsilon=10^{-2}$, $v_1(0)=\varepsilon$, $v_2(0)=10^{-6}$, $v_3(0)=0.98$. 

\noindent   Figure \ref{fig:phase-trajectory} (solution of \eqref{eq:main-1}): $\varepsilon=10^{-2}$, $v_1(0)=\varepsilon$, $v_2(0)=10^{-6}$, $u_3(0)=0.98$.

\noindent   Figure \ref{Figure3} (solution of \eqref{eq:system}): $\varepsilon=10^{-2}$, $u_1(0)=1$, $u_2(0)=10^{-3}$, $u_3(0)=0.98$.

\bibliographystyle{plain}
\bibliography{biblio.bib}
\end{document}